\documentclass[11pt]{amsart}
\usepackage{amsthm}
\usepackage{amsmath,amstext,amsthm,amsfonts,amssymb,amsthm}
\usepackage{mathrsfs}
\usepackage{multicol}
\usepackage{latexsym}
\usepackage{color}
\usepackage{graphicx}
\usepackage{epsf}
\usepackage{epsfig}
\usepackage{epic}
\usepackage{graphics}
\usepackage{verbatim}
\usepackage{color}
\usepackage{hyperref}
\usepackage{bm}
\newtheorem{theorem}{Theorem}[section]
\newtheorem{lemma}[theorem]{Lemma}
\newtheorem{corollary}[theorem]{Corollary}
\theoremstyle{definition}
\newtheorem{definition}[theorem]{Definition}
\newtheorem{proposition}[theorem]{Proposition}

\theoremstyle{remark}
\newtheorem{remark}[theorem]{Remark}
\numberwithin{equation}{section}

\newcommand{\q}{\mathfrak{q}}
\newcommand{\HI}{\mathfrak{H}}

\newcommand{\C}{\mathbb{C}}

\newcommand{\B}{\mathcal{B}}

\newcommand{\qu}{\mathfrak{q}}
\newcommand{\pu}{\mathfrak{p}}
\newcommand{\oqu}{\overline{\mathfrak{q}}}

\newcommand{\quat}{\mathbb H}
\newcommand{\R}{\Bbb R}
\newcommand{\Z}{\mathbb Z}
\newcommand{\mc}{\mathcal}

\newcommand{\be}{\begin{equation}}
\newcommand{\en}{\end{equation}}
\newcommand{\D}{{\mc D}}

\newcommand{\FF}{\mc F}

\newcommand{\bedefin}{\begin{defi}}
	\newcommand{\findefi}{\end{defi} \medskip}
\newcommand{\betheo}{\begin{theorem}$\!\!${\bf \,\,\,}}
	\newcommand{\entheo}{\end{theorem}}
\newcommand{\enth}{\end{theorem}}
\newcommand{\becor}{\begin{cor}$\!\!${\bf .}}
	\newcommand{\encor}{\end{cor}}
\newcommand{\belem}{\begin{lem}$\!\!${\bf .}}
	\newcommand{\enlem}{\end{lem}}
\newcommand{\bea}{\begin{eqnarray}}
\newcommand{\ena}{\end{eqnarray}}
\newcommand{\beano}{\begin{eqnarray*}}
	\newcommand{\enano}{\end{eqnarray*}}
\newcommand{\bee}{\begin{enumerate}}
	\newcommand{\ene}{\end{enumerate}}
\newcommand{\bei}{\begin{itemize}}
	\newcommand{\eni}{\end{itemize}}
\newcommand{\betab}{\begin{tabular}}
	\newcommand{\entab}{\end{tabular}}

\newcommand{\Iop}{{\mathbb{I}_{V_{\mathbb{H}}^{R}}}}

\newcommand{\bfraka}{\mbox{\boldmath $\mathfrak a$}}
\newcommand{\bfrakb}{\mbox{\boldmath $\mathfrak b$}}

\newcommand{\bfrakp}{\mbox{ $\mathfrak p$}}

\newcommand{\bk}{\mathbf k}

\newcommand{\bi}{\mathbf i}
\newcommand{\bj}{\mathbf j}

\newcommand{\apo}{\sigma_{ap}^S}
\newcommand{\vr}{V_\quat^R}
\newcommand{\ur}{U_\quat^R}
\newcommand{\ra}{\text{ran}}
\newcommand{\kr}{\text{ker}}
\newcommand{\ind}{\text{ind}}
\newcommand{\cokr}{\text{coker}}
\newcommand{\Iopu}{\mathbb{I}_{\ur}}
\newcommand{\sel}{\sigma^S_{el}}
\newcommand{\ser}{\sigma^S_{er}}
\newcommand{\se}{\sigma_e^S}
\newcommand{\Wy}{\mathcal{W}}

\begin{document}
\title[Fredholm operator]{Fredholm operators and essential S-spectrum in the quaternionic setting}
\author{B. Muraleetharan$^{\dagger}$ and 
K. Thirulogasanthar$^{\ddagger}$}
\address{$^{\dagger}$ Department of mathematics and Statistics, University of Jaffna, Thirunelveli, Sri Lanka.}
\address{$^{\ddagger}$ Department of Computer Science and Software Engineering, Concordia University, 1455 De Maisonneuve Blvd. West, Montreal, Quebec, H3G 1M8, Canada.}
\email{bbmuraleetharan@jfn.ac.lk and santhar@gmail.com  
}
\date{\today}
\date{\today}
\begin{abstract}
For bounded right linear operators, in a right quaternionic Hilbert space with a left multiplication defined on it, we study the approximate $S$-point spectrum. In the same Hilbert space, then we study the Fredholm operators and the Fredholm index. In particular, we prove the invariance of the Fredholm index under small norm operator and compact operator perturbations. Finally, in association with the Fredholm operators, we develop the theory of essential S-spectrum. We also characterize the $S$-spectrum in terms of the essential S-spectrum and Fredholm operators. In the sequel we study left and right S-spectrums as needed for the development of the theory presented in this note.
\end{abstract}
\maketitle
\pagestyle{myheadings}
\section{Introduction}
In the complex theory Fredholm operators play an important role in the investigations of various classes of singular integral equations, in the theory of perturbations of Hermitian operators by Hermitian and non-Hermitian operators, and in obtaining a priori estimate in determining properties of certain differential operators \cite{Wil, Jeri, con, Goh, kub}.  Fredholm alternative theorem is used to derive an adjoint equation to the linear stability equations in Fluid dynamics, and useful in scattering of a 1-D particle on a small potential barrier, see  \cite{Jeri} and the many references therein.\\

In the complex case, studies of Fredholm theory and perturbation results are of a great importance in the description of the essential spectrum. In several applications, such as essential spectrum of the Schr\"odinger equations, essential spectrum of the perturbed Hamiltonians, in general, essential spectrum of the differential operators, essential spectrum of the transport operator (transport of the neutron, photons, molecules in gas, etc),  in particular, the time dependent transport equations arise in a number of different applications in Biology, Chemistry and Physics, information about it is important \cite{Ben, Goh, Jeri, Jeri2, Gar, Wil}. Further, near the essential spectrum, numerical calculations of eigenvalues become difficult. Hence they have to be treated analytically \cite{Jeri}.  There are several distinct definitions of the essential spectrum. However they all coincide for the self-adjoint operators on Hilbert spaces \cite{Goh, Ben}. In this note we only consider the essential spectrum associated with Fredholm operators in the quaternionic setting.\\

In the complex setting, in a Hilbert space $\HI$, a bounded linear operator, $A$, is not invertible if it is not bounded below (same is true in the quaternion setting, see theorem \ref{NT1} below). The set of approximate eigenvalues which are $\lambda\in\C$ such that $A-\lambda \mathbb{I}_\HI$, where $\mathbb{I}_{\HI}$ is the identity operator on $\HI$,  is not bounded below, equivalently, the set of $\lambda\in\C$ for which there is a sequence of unit vectors $\phi_1,\phi_2,\cdots$ such that $\displaystyle\lim_{n\rightarrow\infty}\|A\phi_n-\lambda\phi_n\|=0$. The set of approximate eigenvalues is known as the approximate spectrum. In the quaternionic setting, let $\vr$ be a separable right Hilbert space,  $A$ be a bounded right linear operator, and $R_\qu(A)=A^2-2\text{Re}(\qu)A+|\qu|^2\Iop$, with $\qu\in\quat$, the set of all quaternions, be the pseudo-resolvent operator, the set of right eigenvalues of $R_\qu(A)$ coincide with the point S-spectrum (see proposition 4.5 in \cite{ghimorper}). In this regard, it will be appropriate to define and study the quaternionic approximate S-point spectrum as the quaternions for which $R_\qu(A)$ in not bounded below.\\

Due to the non-commutativity, in the quaternionic case  there are three types of  Hilbert spaces: left, right, and two-sided, depending on how vectors are multiplied by scalars. This fact can entail several problems. For example, when a Hilbert space $\mathcal H$ is one-sided (either left or right) the set of linear operators acting on it does not have a linear structure. Moreover, in a one sided quaternionic Hilbert space, given a linear operator $A$ and a quaternion $\mathfrak{q}\in\quat$, in general we have that $(\mathfrak{q} A)^{\dagger}\not=\overline{\mathfrak{q}} A^{\dagger}$ (see \cite{Mu} for details). These restrictions can severely prevent the generalization  to the quaternionic case of results valid in the complex setting. Even though most of the linear spaces are one-sided, it is possible to introduce a notion of multiplication on both sides by fixing an arbitrary Hilbert basis of $\mathcal H$.  This fact allows to have a linear structure on the set of linear operators, which is a minimal requirement to develop a full theory. Thus, the framework of this paper, is in part, is a right quaternionic Hilbert space equipped with a left multiplication, introduced by fixing a Hilbert basis.\\

As far as we know, the Fredholm operator theory and the essential S-spectrum and the approximate S-point spectrum have not been studied in the quaternionic setting yet. In this regard, in this note we investigate the quaternionic S-point spectrum, Fredholm operators and associated S-essential spectrum for a bounded right linear operator on a right quaternionic separable Hilbert space. Since the pseudo-resolvent operator, $R_\qu(A)$ has real coefficients the left multiplication defined on a right quaternionic Hilbert space play a little role. Even the non-commutativity of quaternions does not play an essential role. Even though the S-approximate point spectrum, the S-essential spectrum and the Fredholm operators are structurally different from its complex counterparts, the results we obtain and their proofs are somewhat similar to those in the corresponding complex theory.\\

The article is organized as follows. In section 2 we introduce the set of quaternions and quaternionic Hilbert spaces and their bases, as needed for the development of this article, which may not be familiar to a broad range of audience. In section 3 we define and investigate, as needed, right linear operators and their properties. We have given proofs for some results which are not available in the literature. In section 3.1 we define a basis dependent left multiplication on a right quaternionic Hilbert space. In section 3.2 we deal with the right S-spectrum, left S-spectrum, S-spectrum and its major partitions. In section 4 we provide a systematic study of compact operators which has not yet been done in the literature. In section 5 we study the approximate S-point spectrum, $\sigma_{ap}^S(A)$, of a bounded right linear operator, $A,$ on a right quaternionic Hilbert space. In particular we prove that $\sigma_{ap}^S(A)$ is a non-empty closed subset of $\quat$ and the S-spectrum is the union of the $\sigma_{ap}^S(A)$ and the continuous S-spectrum. In section 6 we study the Fredholm operators and its index for a bounded right linear operator. In particular, we prove the invariance of the Fredholm index under small norm operator and compact operator perturbations.  In this section, since a quaternionic multiple of an operator is not involved, the proofs of these results are almost verbatim copies of the complex ones. In section 7, we study the essential S-spectrum as the S-spectrum of the quotient map image of a bounded right linear operator on the quaternionic version of the Calkin algebra and then characterize the S-essential spectrum in terms of Fredholm operators. We also establish the so-called Atkinson’s theorem, prove the invariance of the essential S-spectrum under compact perturbations, and give a characterization to S-spectrum in terms of Fredholm operators and its index (see proposition \ref{EP7}). Section 8 ends the manuscript with a conclusion.

\section{Mathematical preliminaries}
In order to make the paper self-contained, we recall some facts about quaternions which may not be well-known.  For details we refer the reader to \cite{Ad,ghimorper,Vis}.
\subsection{Quaternions}
Let $\quat$ denote the field of all quaternions and $\quat^*$ the group (under quaternionic multiplication) of all invertible quaternions. A general quaternion can be written as
$$\qu = q_0 + q_1 \bi + q_2 \bj + q_3 \bk, \qquad q_0 , q_1, q_2, q_3 \in \mathbb R, $$
where $\bi,\bj,\bk$ are the three quaternionic imaginary units, satisfying
$\bi^2 = \bj^2 = \bk^2 = -1$ and $\bi\bj = \bk = -\bj\bi,  \; \bj\bk = \bi = -\bk\bj,
\; \bk\bi = \bj = - \bi\bk$. The quaternionic conjugate of $\qu$ is
$$ \overline{\qu} = q_0 - \bi q_1 - \bj q_2 - \bk q_3 , $$
while $\vert \qu \vert=(\qu \overline{\qu})^{1/2} $ denotes the usual norm of the quaternion $\qu$.
If $\qu$ is non-zero element, it has inverse
$
\qu^{-1} =  \dfrac {\overline{\qu}}{\vert \qu \vert^2 }.$
Finally, the set
\begin{eqnarray*}
\mathbb{S}&=&\{I=x_1 \bi+x_2\bj+x_3\bk~\vert
~x_1,x_2,x_3\in\mathbb{R},~x_1^2+x_2^2+x_3^2=1\},
\end{eqnarray*}
contains all the elements whose square is $-1$. It is a $2$-dimensional sphere in $\mathbb H$ identified with $\mathbb R^4$.
\subsection{Quaternionic Hilbert spaces}
In this subsection we  discuss right quaternionic Hilbert spaces. For more details we refer the reader to \cite{Ad,ghimorper,Vis}.
\subsubsection{Right quaternionic Hilbert Space}
Let $V_{\quat}^{R}$ be a vector space under right multiplication by quaternions.  For $\phi,\psi,\omega\in V_{\quat}^{R}$ and $\qu\in \quat$, the inner product
$$\langle\cdot\mid\cdot\rangle_{V_{\quat}^{R}}:V_{\quat}^{R}\times V_{\quat}^{R}\longrightarrow \quat$$
satisfies the following properties
\begin{enumerate}
	\item[(i)]
	$\overline{\langle \phi\mid \psi\rangle_{V_{\quat}^{R}}}=\langle \psi\mid \phi\rangle_{V_{\quat}^{R}}$
	\item[(ii)]
	$\|\phi\|^{2}_{V_{\quat}^{R}}=\langle \phi\mid \phi\rangle_{V_{\quat}^{R}}>0$ unless $\phi=0$, a real norm
	\item[(iii)]
	$\langle \phi\mid \psi+\omega\rangle_{V_{\quat}^{R}}=\langle \phi\mid \psi\rangle_{V_{\quat}^{R}}+\langle \phi\mid \omega\rangle_{V_{\quat}^{R}}$
	\item[(iv)]
	$\langle \phi\mid \psi\qu\rangle_{V_{\quat}^{R}}=\langle \phi\mid \psi\rangle_{V_{\quat}^{R}}\qu$
	\item[(v)]
	$\langle \phi\qu\mid \psi\rangle_{V_{\quat}^{R}}=\overline{\qu}\langle \phi\mid \psi\rangle_{V_{\quat}^{R}}$
\end{enumerate}
where $\overline{\qu}$ stands for the quaternionic conjugate. It is always assumed that the
space $V_{\quat}^{R}$ is complete under the norm given above and separable. Then,  together with $\langle\cdot\mid\cdot\rangle$ this defines a right quaternionic Hilbert space. Quaternionic Hilbert spaces share many of the standard properties of complex Hilbert spaces. In this paper, more than one Hilbert spaces will be involved in the results, but one notices that every Hilbert space, involved in results, is right quaternionic Hilbert space.

The next two Propositions can be established following the proof of their complex counterparts, see e.g. \cite{ghimorper,Vis}.
\begin{proposition}\label{P1}
Let $\mathcal{O}=\{\varphi_{k}\,\mid\,k\in N\}$
be an orthonormal subset of $V_{\quat}^{R}$, where $N$ is a countable index set. Then following conditions are pairwise equivalent:
\begin{itemize}
\item [(a)] The closure of the linear combinations of elements in $\mathcal O$ with coefficients on the right is $V_{\quat}^{R}$.
\item [(b)] For every $\phi,\psi\in V_{\quat}^{R}$, the series $\sum_{k\in N}\langle\phi\mid\varphi_{k}\rangle_{V_{\quat}^{R}}\langle\varphi_{k}\mid\psi\rangle_{V_{\quat}^{R}}$ converges absolutely and it holds:
$$\langle\phi\mid\psi\rangle_{V_{\quat}^{R}}=\sum_{k\in N}\langle\phi\mid\varphi_{k}\rangle_{V_{\quat}^{R}}\langle\varphi_{k}\mid\psi\rangle_{V_{\quat}^{R}}.$$
\item [(c)] For every  $\phi\in V_{\quat}^{R}$, it holds:
$$\|\phi\|^{2}_{V_{\quat}^{R}}=\sum_{k\in N}\mid\langle\varphi_{k}\mid\phi\rangle_{V_{\quat}^{R}}\mid^{2}.$$
\item [(d)] $\mathcal{O}^{\bot}=\{0\}$.
\end{itemize}
\end{proposition}
\begin{definition}
The set $\mathcal{O}$ as in Proposition \ref{P1} is called a {\em Hilbert basis} of $V_{\quat}^{R}$.
\end{definition}
\begin{proposition}\label{P2}
Every quaternionic Hilbert space $V_{\quat}^{R}$ has a Hilbert basis. All the Hilbert bases of $V_{\quat}^{R}$ have the same cardinality.

Furthermore, if $\mathcal{O}$ is a Hilbert basis of $V_{\quat}^{R}$, then every  $\phi\in V_{\quat}^{R}$ can be uniquely decomposed as follows:
$$\phi=\sum_{k\in N}\varphi_{k}\langle\varphi_{k}\mid\phi\rangle_{V_{\quat}^{R}},$$
where the series $\sum_{k\in N}\varphi_k\langle\varphi_{k}\mid\phi\rangle_{V_{\quat}^{R}}$ converges absolutely in $V_{\quat}^{R}$.
\end{proposition}

It should be noted that once a Hilbert basis is fixed, every left (resp. right) quaternionic Hilbert space also becomes a right (resp. left) quaternionic Hilbert space \cite{ghimorper,Vis}.

The field of quaternions $\quat$ itself can be turned into a left quaternionic Hilbert space by defining the inner product $\langle \qu \mid \qu^\prime \rangle = \qu \overline{\qu^{\prime}}$ or into a right quaternionic Hilbert space with  $\langle \qu \mid \qu^\prime \rangle = \overline{\qu}\qu^\prime$.
\section{Right quaternionic linear  operators and some basic properties}
In this section we shall define right  $\quat$-linear operators and recall some basis properties. Most of them are very well known. In this manuscript, we follow the notations in \cite{AC} and \cite{ghimorper}. We shall also recall some results pertinent to the development of the paper. 
\begin{definition}
A mapping $A:\D(A)\subseteq V_{\quat}^R \longrightarrow U_{\quat}^R$, where $\D(A)$ stands for the domain of $A$, is said to be right $\quat$-linear operator or, for simplicity, right linear operator, if
$$A(\phi\bfraka+\psi\bfrakb)=(A\phi)\bfraka+(A\psi)\bfrakb,~~\mbox{~if~}~~\phi,\,\psi\in \D(A)~~\mbox{~and~}~~\bfraka,\bfrakb\in\quat.$$
\end{definition}
The set of all right linear operators from $V_{\quat}^{R}$ to $U_{\quat}^{R}$ will be denoted by $\mathcal{L}(V_{\quat}^{R},U_{\quat}^{R})$ and the identity linear operator on $V_{\quat}^{R}$ will be denoted by $\Iop$. For a given $A\in \mathcal{L}(V_{\quat}^{R},U_{\quat}^{R})$, the range and the kernel will be
\begin{eqnarray*}
\text{ran}(A)&=&\{\psi \in U_{\quat}^{R}~|~A\phi =\psi \quad\text{for}~~\phi \in\D(A)\}\\
\ker(A)&=&\{\phi \in\D(A)~|~A\phi =0\}.
\end{eqnarray*}
We call an operator $A\in \mathcal{L}(V_{\quat}^{R},U_{\quat}^{R})$ bounded if
\begin{equation}\label{PE1}
\|A\|=\sup_{\|\phi \|_{\vr}=1}\|A\phi \|_{\ur}<\infty,
\end{equation}
or equivalently, there exist $K\geq 0$ such that $\|A\phi \|_{\ur}\leq K\|\phi \|_{\vr}$ for all $\phi \in\D(A)$. The set of all bounded right linear operators from $V_{\quat}^{R}$ to $U_{\quat}^{R}$ will be denoted by $\B(V_{\quat}^{R},U_{\quat}^{R})$. Set of all  invertible bounded right linear operators from $V_{\quat}^{R}$ to $U_{\quat}^{R}$ will be denoted by $\mathcal{G} (V_{\quat}^{R},U_{\quat}^{R})$. We also denote for a set $\Delta\subseteq\quat$, $\Delta^*=\{\oqu~|~\qu\in\Delta\}$.
\\
Assume that $V_{\quat}^{R}$ is a right quaternionic Hilbert space, $A$ is a right linear operator acting on it.
Then, there exists a unique linear operator $A^{\dagger}$ such that
\begin{equation}\label{Ad1}
\langle \psi \mid A\phi \rangle_{\ur}=\langle A^{\dagger} \psi \mid\phi \rangle_{\vr};\quad\text{for all}~~~\phi \in \D (A), \psi\in\D(A^\dagger),
\end{equation}
where the domain $\D(A^\dagger)$ of $A^\dagger$ is defined by
$$
\D(A^\dagger)=\{\psi\in U_{\quat}^{R}\ |\ \exists \varphi\ {\rm such\ that\ } \langle \psi \mid A\phi \rangle_{\ur}=\langle \varphi \mid\phi \rangle_{\vr}\}.$$
The following theorem gives two important and fundamental results about right $\quat$-linear bounded operators which are already appeared in \cite{ghimorper} for the case of $\vr=\ur$. Point (b) of the following theorem is known as the open mapping theorem.
\begin{theorem}\label{open} Let $A:\D(A)\subseteq V_{\quat}^R \longrightarrow U_{\quat}^R$ be a right $\quat$-linear operator. Then
\begin{itemize} 
\item[(a)] $A\in\B(\vr,\ur)$ if and only if $A$ is continuous.
\item[(b)] if $A\in\B(\vr,\ur)$ is surjective, then $A$ is open. In particular, if $A$ is bijective then $A^{-1}\in\B(\vr,\ur)$.
\end{itemize}
\end{theorem}
\begin{proof}
The proof is the same as the proof in a complex Hilbert space, (see e.g. \cite{con}).
\end{proof}
The following proposition provides some useful aspects about the orthogonal complement subsets.
\begin{proposition}\label{ort}
Let $M\subseteq\vr$. Then
\begin{itemize}
\item [(a)] $M^{^\perp}$ is closed.
\item [(b)] if $M$ is a closed subspace of $\vr$ then $\vr=M\oplus M^\perp$.
\item [(c)] if $\dim(M)<\infty$, then $M$ is a closed subspace.
\end{itemize}
\end{proposition}
\begin{proof}
The proof is similar to the proof of the complex case, (see e.g. \cite{Rud}).
\end{proof}
The points (a) and (b) of the following proposition are already appeared in \cite{ghimorper} for the case $\vr=\ur$.
\begin{proposition}\label{IP30}
Let $A\in\B(\vr, \ur)$. Then
\begin{itemize}
\item [(a)] $\ra(A)^\perp=\kr(A^\dagger).$
\item [(b)] $\kr(A)=\ra(A^\dagger)^\perp.$
\item [(c)] $\kr(A)$ is closed subspace of $\vr$.
\end{itemize}
\end{proposition}
\begin{proof}
The proofs are elementary.
\end{proof}
\begin{proposition}\label{IP4}
Let $A:\D(A):\vr\longrightarrow\ur$ be a right quaternionic linear operator. If $A$ is closed and satisfies the condition that there exists $c>0$ such that  \begin{center}
$\|A\phi\|_{\ur}\geq c\|\phi\|_{\vr}$, for all $\phi\in\D(A)$,
\end{center} then $\text{ran}(A)$ is closed.
\end{proposition}
\begin{proof}
The proof can be manipulated from the proof of proposition 2.13 in \cite{BT}.
\end{proof}
\begin{proposition}\label{PP3}
If $A\in\B(\vr)$ and if $\|A\|<1$, then the right linear operator $\Iop-A$ is invertible and the inverse is given by $\displaystyle(\Iop-A)^{-1}=\sum_{k=0}^{n}A^k.$
\end{proposition}
\begin{proof}
The proof is exactly the same as its complex version, see theorem 2.17, in \cite{Tok} for a complex proof.
\end{proof}
\begin{theorem}(Bounded inverse theorem)\label{NT1}
Let $A\in\B(\vr,\ur)$, then the following results are equivalent.
\begin{enumerate}
\item [(a)] $A$ has a bounded inverse on its range.
\item[(b)] $A$ is bounded below.
\item[(c)] $A$ is injective and has a closed range.
\end{enumerate}
\end{theorem}
\begin{proof}
The proof is exactly similar to its complex version. See theorem 1.2 in \cite{kub} for a complex proof.
\end{proof}
\begin{proposition}\label{NP1}
Let $A\in\B(\vr,\ur)$, then $\ra(A)$ is closed in $\ur$ if and only if $\ra(A^\dagger)$ is closed in $\vr$.
\end{proposition}
\begin{proof}
The proof is the same as for the complex
Hilbert spaces, see lemma 1.5 in \cite{kub} for a complex proof.
\end{proof}
\begin{proposition}\label{NP2}
Let $A\in\B(\vr)$. Then,
\begin{enumerate}
\item [(a)] $A$ is invertible if and only if it is injective with a closed range (i.e., $\kr(A)=\{0\}$ and $\overline{\ra(A)}=\ra(A)$).
\item[(b)] $A$ is left (right) invertible if and only if $A^\dagger$ is right (left) invertible.
\item[(c)] $A$ is right invertible if and only if it is surjective (i.e., $\ra(A)=\vr$).
\end{enumerate}
\end{proposition}
\begin{proof}
The proof is also exactly similar to its complex version. See lemma 5.8 in \cite{kub} for a complex proof.
\end{proof}
\begin{definition}\label{ID1}
Let $A\in\B(\vr)$. A closed subspace $M\subseteq\vr$ is said to be invariant under $A$ if $A(M)\subseteq M$, where $A(M)=\{A\phi~|~\phi\in M\}$.
\end{definition}
Following results hold in any quaternionic normed linear space.
\begin{definition}
Let $X_\quat^R$ be a quaternionic normed linear space. A subset $C$ is said to be totally bounded, if for every $\varepsilon>0$, there exist $N\in\mathbb{N}$ and $\varphi_i\in C:\,i=1,2,3,\cdots,N$ such that 
$$C\subseteq\bigcup_{i=1}^{N}B(\varphi_i;\varepsilon),$$ where $B(\varphi_i;\varepsilon)=\{\phi\in X_\quat^R~\mid~\|\phi-\varphi_i\|_{X_\quat^R}<\varepsilon\}$ for all $i=1,2,3,\cdots,N$.
\end{definition}
\begin{proposition}\label{Com}
Let $C$ be a subset in a right quaternionic normed linear space $X_\quat^R$. Then
\begin{itemize}
\item[(a)] $C$ is compact if and only if $C$ is complete and totally bounded. 
\item[(b)] if $\dim(X_\quat^R)<\infty$,  $C$ is compact if and only if $C$ is closed and bounded.
\end{itemize}
\end{proposition}
\begin{proof}
The proof is exactly similar to the complex proof. For a complex  proof  of (a) and (b), see theorem 6.19 in \cite{Mus} and theorem 2.5-3 in \cite{Kre} respectively. 
\end{proof}
\subsection{Left Scalar Multiplications on $\ur$.}
We shall extract the definition and some properties of left scalar multiples of vectors on $\ur$ from \cite{ghimorper} as needed for the development of the manuscript. The left scalar multiple of vectors on a right quaternionic Hilbert space is an extremely non-canonical operation associated with a choice of preferred Hilbert basis. From the proposition \ref{P2}, $\ur$ has a Hilbert basis
\begin{equation}\label{b1}
\mathcal{O}=\{\varphi_{k}\,\mid\,k\in N\},
\end{equation}
where $N$ is a countable index set.
The left scalar multiplication on $\ur$ induced by $\mathcal{O}$ is defined as the map $\quat\times \ur\ni(\qu,\phi)\longmapsto \qu\phi\in \ur$ given by
\begin{equation}\label{LPro}
\qu\phi:=\sum_{k\in N}\varphi_{k}\qu\langle \varphi_{k}\mid \phi\rangle,
\end{equation}
for all $(\qu,\phi)\in\quat\times \ur$.
\begin{proposition}\cite{ghimorper}\label{lft_mul}
The left product defined in the equation \ref{LPro} satisfies the following properties. For every $\phi,\psi\in \ur$ and $\bfrakp,\qu\in\quat$,
\begin{itemize}
\item[(a)] $\qu(\phi+\psi)=\qu\phi+\qu\psi$ and $\qu(\phi\bfrakp)=(\qu\phi)\bfrakp$.
\item[(b)] $\|\qu\phi\|=|\qu|\|\phi\|$.
\item[(c)] $\qu(\bfrakp\phi)=(\qu\bfrakp)\phi$.
\item[(d)] $\langle\overline{\qu}\phi\mid\psi\rangle
=\langle\phi\mid\qu\psi\rangle$.
\item[(e)] $r\phi=\phi r$, for all $r\in \mathbb{R}$.
\item[(f)] $\qu\varphi_{k}=\varphi_{k}\qu$, for all $k\in N$.
\end{itemize}
\end{proposition}
\begin{remark}
(1) The meaning of writing $\bfrakp\phi$ is $\bfrakp\cdot\phi$, because the notation from the equation \ref{LPro} may be confusing, when $\ur=\quat$. However, regarding the field $\quat$ itself as a right $\quat$-Hilbert space, an orthonormal basis $\mathcal{O}$ should consist only of a singleton, say $\{\varphi_{0}\}$, with $\mid\varphi_{0}\mid=1$, because we clearly have $\theta=\varphi_{0}\langle\varphi_{0}\mid\theta\rangle$, for all $\theta\in\quat$. The equality from (f) of proposition \ref{lft_mul} can be written as $\bfrakp\varphi_{0}=\varphi_{0}\bfrakp$, for all $\bfrakp\in\quat$. In fact, the left hand may be confusing and it should be understood as $\bfrakp\cdot\varphi_{0}$, because the true equality $\bfrakp\varphi_{0}=\varphi_{0}\bfrakp$ would imply that $\varphi_{0}=\pm 1$. For the simplicity, we are writing $\bfrakp\phi$ instead of writing $\bfrakp\cdot\phi$.\\
(2) Also one can trivially see that $(\bfrakp+\qu)\phi=\bfrakp\phi+\qu\phi$, for all $\bfrakp,\qu\in\quat$ and $\phi\in \ur$.
\end{remark}
Furthermore, the quaternionic left scalar multiplication of linear operators is also defined in \cite{Fab1}, \cite{ghimorper}. For any fixed $\qu\in\quat$ and a given right linear operator $A:\vr\longrightarrow \ur$, the left scalar multiplication of $A$ is defined as a map $\qu A:\vr\longrightarrow \ur$ by the setting
\begin{equation}\label{lft_mul-op}
(\qu A)\phi:=\qu (A\phi)=\sum_{k\in N}\varphi_{k}\qu\langle \varphi_{k}\mid A\phi\rangle,
\end{equation}
for all $\phi\in \vr)$. It is straightforward that $\qu A$ is a right linear operator. If $\qu\phi\in \vr$, for all $\phi\in \vr$, one can define right scalar multiplication of the right linear operator $A:\vr\longrightarrow \ur$ as a map $ A\qu:\vr\longrightarrow \ur$ by the setting
\begin{equation}\label{rgt_mul-op}
(A\qu)\phi:=A(\qu \phi),
\end{equation}
for all $\phi\in D(A)$. It is also right linear operator. One can easily obtain that, if $\qu\phi\in \vr$, for all $\phi\in \vr$ and $\vr$ is dense in $\ur$, then
\begin{equation}\label{sc_mul_aj-op}
(\qu A)^{\dagger}=A^{\dagger}\overline{\qu}~\mbox{~and~}~
(A\qu)^{\dagger}=\overline{\qu}A^{\dagger}.
\end{equation}
\subsection{S-Spectrum}
For a given right linear operator $A:\D(A)\subseteq V_{\quat}^R\longrightarrow V_{\quat}^R$ and $\qu\in\quat$, we define the operator $R_{\qu}(A):\D(A^{2})\longrightarrow\quat$ by  $$R_{\qu}(A)=A^{2}-2\text{Re}(\qu)A+|\qu|^{2}\Iop,$$
where $\qu=q_{0}+\bi q_1 + \bj q_2 + \bk q_3$ is a quaternion, $\text{Re}(\qu)=q_{0}$  and $|\qu|^{2}=q_{0}^{2}+q_{1}^{2}+q_{2}^{2}+q_{3}^{2}.$\\
In the literature, the operator is called pseudo-resolvent since it is not the resolvent operator of $A$ but it is the one related to the notion of spectrum as we shall see in the next definition. For more information, on the notion of $S$-spectrum the reader may consult e.g. \cite{Fab, Fab1, NFC}, and  \cite{ghimorper}.
\begin{definition}
Let $A:\D(A)\subseteq V_{\quat}^R\longrightarrow V_{\quat}^R$ be a right linear operator. The {\em $S$-resolvent set} (also called \textit{spherical resolvent} set) of $A$ is the set $\rho_{S}(A)\,(\subset\quat)$ such that the three following conditions hold true:
\begin{itemize}
\item[(a)] $\ker(R_{\qu}(A))=\{0\}$.
\item[(b)] $\text{ran}(R_{\qu}(A))$ is dense in $V_{\quat}^{R}$.
\item[(c)] $R_{\qu}(A)^{-1}:\text{ran}(R_{\qu}(A))\longrightarrow\D(A^{2})$ is bounded.
\end{itemize}
The \textit{$S$-spectrum} (also called \textit{spherical spectrum}) $\sigma_{S}(A)$ of $A$ is defined by setting $\sigma_{S}(A):=\quat\smallsetminus\rho_{S}(A)$. For a bounded linear operator $A$ we can write the resolvent set as
\begin{eqnarray*}
\rho_S(A)&=& \{\qu\in\quat~|~R_\qu(A)\in\mathcal{G}(V_{\quat}^R)\}\\
&=&\{\qu\in\quat~|~R_\qu(A)~\text{has an inverse in}~\B(V_{\quat}^R)\}\\
&=&\{\qu\in\quat~|~\text{ker}(R_\qu(A))=\{0\}\quad\text{and}\quad \text{ran}(R_\qu(A))=V_\quat^R\}
\end{eqnarray*}
and the spectrum can be written as
\begin{eqnarray*}
\sigma_S(A)&=&\quat\setminus\rho_S(A)\\
&=&\{\qu\in\quat~|~R_\qu(A)~\text{has no inverse in}~\B(V_{\quat}^R)\}\\
&=&\{\qu\in\quat~|~\text{ker}(R_\qu(A))\not=\{0\}\quad\text{or}\quad \text{ran}(R_\qu(A))\not=V_\quat^R\}
\end{eqnarray*}
The right $S$-spectrum $\sigma_r^S(A)$ and the left $S$-spectrum $\sigma_l^S(A)$ are defined respectively as
\begin{eqnarray*}\
\sigma_r^S(A)&=&\{\qu\in\quat~|~R_\qu(A)~~\text{in not right invertible in}~~\B(V_\quat^R)~\}\\
\sigma_l^S(A)&=&\{\qu\in\quat~|~R_\qu(A)~~\text{in not left invertible in}~~\B(V_\quat^R)~\}.
\end{eqnarray*}
The spectrum $\sigma_S(A)$ decomposes into three major disjoint subsets as follows:
\begin{itemize}
\item[(i)] the \textit{spherical point spectrum} of $A$: $$\sigma_{pS}(A):=\{\qu\in\quat~\mid~\ker(R_{\qu}(A))\ne\{0\}\}.$$
\item[(ii)] the \textit{spherical residual spectrum} of $A$: $$\sigma_{rS}(A):=\{\qu\in\quat~\mid~\ker(R_{\qu}(A))=\{0\},\overline{\text{ran}(R_{\qu}(A))}\ne V_{\quat}^{R}~\}.$$
\item[(iii)] the \textit{spherical continuous spectrum} of $A$: $$\sigma_{cS}(A):=\{\qu\in\quat~\mid~\ker(R_{\qu}(A))=\{0\},\overline{\text{ran}(R_{\qu}(A))}= V_{\quat}^{R}, R_{\qu}(A)^{-1}\notin\B(V_{\quat}^{R}) ~\}.$$
\end{itemize}
If $A\phi=\phi\qu$ for some $\qu\in\quat$ and $\phi\in V_{\quat}^{R}\smallsetminus\{0\}$, then $\phi$ is called an \textit{eigenvector of $A$ with right eigenvalue} $\qu$. The set of right eigenvalues coincides with the point $S$-spectrum, see \cite{ghimorper}, proposition 4.5.
\end{definition}
\begin{proposition}\cite{Fab2, ghimorper}\label{PP1}
For $A\in\B(\vr)$, the resolvent set $\rho_S(A)$ is a non-empty open set and the spectrum $\sigma_S(A)$ is a non-empty compact set.
\end{proposition}
\begin{remark}\label{R1}
For $A\in\B(\vr)$, since $\sigma_S(A)$ is a non-empty compact set so is its boundary. That is, $\partial\sigma_S(A)=\partial\rho_S(A)\not=\emptyset$.
\end{remark}
\begin{proposition}\label{SP1}
Let $A\in\B(\vr)$.
\begin{eqnarray}
\sigma_l^S(A)&=&\{\qu\in\quat~~|~~\ra(R_\qu(A))~~\text{is closed or}~~\kr(R_\qu(A))\not=\{0\}\}\label{SE1}.\\
\sigma_r^S(A)&=&\{\qu\in\quat~~|~~\ra(R_\qu(A))~~\text{is closed or}~~\kr(R_{\oqu}(A^\dagger))\not=\{0\}\}\label{SE2}.
\end{eqnarray}
\end{proposition}
\begin{proof}
Let $A\in\B(\vr)$ and $\qu\in\quat$. Set $S=R_\qu(A)\in\B(\vr)$. By proposition \ref{NP2}, $S$ is not left invertible if and only if $\ra(S)$ is not closed or $\kr(S)\not=\{0\}$. Thus we have equation \ref{SE1}. Again by proposition \ref{NP2}, $A$ is not right invertible if and only if $\ra(S)\not=\vr$. $\ra(S)\not=\vr$ if and only if $\ra(S)\not=\overline{\ra(S)}$ or $\overline{\ra(S)}\not=\vr$. That is, $\ra(S)\not=\vr$ if and only if $\ra(S)\not=\overline{\ra(S)}$ or $\kr(S^\dagger)^\perp\not=\vr$. Hence we have equation \ref{SE2}. 
\end{proof}
\section{Quaternionic Compact Operators}
A systematic study of quaternionic compact operators has not appeared in the literature. In this regard, in this section, as needed for our purpose, we provide certain significant results about compact right linear operators on $\vr$.
\begin{definition}\label{PD1}
Let $\vr$ and $\ur$ be right quaternionic  Hilbert spaces. A bounded operator $K:\vr\longrightarrow\ur$ is compact if $K$ maps bounded sets into precompact sets. That is, $\overline{K(U)}$ is compact in $\ur$, where $U=\{\phi\in\vr~|~\|\phi\|_{\vr}\leq1\}$. Equivalently, for all bounded sequences $\{\phi_n\}_{n=1}^\infty$ in $\vr$ the sequence $\{K\phi_n\}_{n=0}^\infty$ has a convergence subsequence in $\ur$ \cite{Fa}.
\end{definition}
We denote the set of all compact operators from $\vr$ to $\ur$ by $\B_0(\vr,\ur)$ and the compact operators from $\vr$ from $\vr$ will be denoted by $\B_0(\vr)$.
\begin{proposition}\label{com1} The following statements are true:
\begin{itemize}
\item [(a)] If $(\qu,A)\in\quat\times\B_0(\vr,\ur)$, then $\qu A,A\qu\in\B_0(\vr,\ur)$ which are defined by the equations \ref{lft_mul-op} and \ref{rgt_mul-op} respectively.
\item [(b)] $\B_0(\vr,\ur)$ is a vector space under left-scalar multiplication.
\item [(c)] If $\{A_n\}\subseteq\B(\vr,\ur)$ and $A\in\B_0(\vr,\ur)$ such that $\|A_n-A\|\longrightarrow 0$, then $A\in\B_0(\vr,\ur)$.
\item [(d)] If $A\in\B(\ur)$, $B\in\B(\ur)$ and $K\in\B_0(\vr,\ur)$, then $AK$ and $KB$ are compact operators.
\end{itemize}
\end{proposition}
\begin{proof}
Let $\{\phi_n\}\subseteq\vr$ be a bounded sequence. Then $\{A\phi_n\}$ has a convergent subsequence. Thus $\{\qu A\phi_n\}$ also has a convergent subsequence, and so $\qu A$ is a compact operator. Now $\{\qu\phi_n\}\subseteq\vr$ is also bounded sequence. Thus $\{ (A\qu)\phi_n\}$ has a convergent subsequence, as $A$ is compact. Therefore point (a) follows. And point (b) is straightforward.\\
Now suppose that $\{A_n\}\subseteq\B_0(\vr,\ur)$ and $A\in\B(\vr,\ur)$ such that $\|A_n-A\|\longrightarrow 0$ and let $\varepsilon>0$. Then by proposition \ref{Com}, it is enough to show that $\overline{A(U)}$ is totally bounded, as $\overline{A(U)}$ is a complete set, where $U=\{\phi\in\vr~|~\|\phi\|_{\vr}\leq1\}$. Let $\varepsilon>0$, then there exists $N\in\mathbb{N}$ such that $\|A-A_n\|<\varepsilon/6$, for all $n>N$. Since $A_n$ is compact, there are $\phi_j\in U:\,j=1,2,3,\cdots,m$ such that 
$$A_n(U)\subseteq\overline{A_n(U)}\subseteq\bigcup_{i=1}^mB(A_n\phi_j;\varepsilon/6);$$ where  $B(A_n\phi_j;\varepsilon/6)=\{A_n\phi~\mid~\|A_n\phi-A_n\phi_j\|_{\ur}<\varepsilon/6\}$, for all $j=1,2,3,\cdots,m$ and $n\in\mathbb{N}$. Let $\psi\in\overline{A(U)}$. Then there exists
a sequence $\{\phi^{(\ell)}\}\subseteq U$ such that $A\phi^{(\ell)}\longrightarrow\psi$ as $\ell\longrightarrow\infty$. So for each $(n,\ell)\in\mathbb{N}^2$, there is a $\phi_j\in U$ such that $\|A_n\phi_j-A_n\phi^{(\ell)}\|_{\ur}<\varepsilon/6$ as $A_n\phi^{(\ell)}\in A_n(U)$, for all $(n,\ell)\in\mathbb{N}^2$. Therefore, for each $(n,\ell)\in\mathbb{N}^2$ with $n>N$,
\begin{eqnarray*}
\|A\phi_j-A\phi^{(\ell)}\|_{\ur}&\leq&\|A\phi_j-A_n\phi_j\|_{\ur}+\|A_n\phi_j-A_n\phi^{(\ell)}\|_{\ur}+\|A_n\phi^{(\ell)}-A\phi^{(\ell)}\|_{\ur}\\
~&<&2\|A-A_n\|+\varepsilon/6\\
~&<&\varepsilon/2.
\end{eqnarray*}
Now taking limit $\ell\longrightarrow\infty$ both side of the inequality $\|A\phi_j-A\phi^{(\ell)}\|_{\ur}<\varepsilon/2$, gives
$\|A\phi_j-\psi\|_{\ur}\leq\varepsilon/2<\varepsilon$ as the norm $\|\cdot\|$ is a continuous maps. Thus $$\displaystyle\overline{A(U)}\subseteq\bigcup_{i=1}^mB(A\phi_j;\varepsilon).$$ This concludes the result of (c).\\
Now suppose that $A\in\B(\ur)$, $B\in\B(\vr)$ and $K\in\B_0(\vr,\ur)$. Then for any bounded sequence $\{\phi_n\}\subseteq\vr$, $\{K\phi_n\}\subseteq\ur$ has a convergent subsequence $\{K\phi_{n_{k}}\}$, as $K$ is compact. Thus $\{AK\phi_{n_{k}}\}$ is a convergent subsequence of $\{AK\phi_n\}$, because $A$ is continuous. That is $AK$ is compact. In a similar fashion, we can obtain $KB$ is compact. Hence the results hold.
\end{proof}
\begin{definition}\label{PD2}
An operator $A:\vr\longrightarrow\ur$ is said to be of finite rank if $\text{ran}(A)\subseteq\ur$ is finite dimensional.
\end{definition}
\begin{lemma}\label{Lfr}
Following statements are equivalent:
\begin{itemize}
\item [(a)] The operator $A:\vr\longrightarrow\ur$ is a finite rank operator.
\item [(b)] For each $\phi\in\vr$, there exist $(v_i,u_i)\in\vr\times\ur:\,i=1,2,3,\cdots,n$ such that 
\be\label{fr} 
\langle u_i\mid u_j\rangle_{\ur}=\delta_{ij}~~\mbox{and}~~A\phi=\sum_{i=1}^{n}u_i\langle v_i\mid \phi\rangle_{\vr},
\en
where $\delta_{ij}$ is Kronecker delta.
\end{itemize}
\end{lemma}
\begin{proof}
Suppose that (a) holds. Then $\dim(\ra(A))<\infty$ and so there exists an orthonormal basis $e_i:\,i=1,2,3,\cdots,n$ for $\ra(A)$. Thus for each $\phi\in\vr$, $$A\phi=\sum_{i=1}^{n}e_i\langle e_i \mid A\phi\rangle_{\ur}=\sum_{i=1}^{n}e_i\langle A^\dagger e_i\mid \phi\rangle_{\vr}.$$ Point (b) follows as taking $u_i=e_i$ and $v_i=A^\dagger e_i$. On the other hand, suppose (b) holds.  Then it immediately follows that $\ra(A)=\text{right span over}~\, \quat\{u_i\,\mid\,i=1,2,3,\cdots,n\}$ as $u_i\in\ra(A)$, for all $i=1,2,3,\cdots,n$. Assume that $$\exists\,\qu_i\in\quat:~i=1,2,3,\cdots,n \mbox{~~such that~~}\sum_{i=1}^{n}u_i\qu_i=0.$$ Since $\sum_{i=1}^{n}(A^\dagger u_i)\qu_i=0$ and there exists $\phi_i\in\vr$ such that $u_i=A\phi_i$, for all $i=1,2,3,\cdots,n$, we have for each $j=1,2,3,\cdots,n$, $$\overline{\qu}_j=\sum_{i=1}^{n}\overline{\qu}_i\langle u_i\mid A\phi_j\rangle_{\ur}=\sum_{i=1}^{n}\langle (A^\dagger u_i)\qu_i\mid \phi_j\rangle_{\vr}=0.$$ That is, the set $\{u_i\,\mid\,i=1,2,3,\cdots,n\}$ is linearly independent. Thus $\{u_i\,\mid\,i=1,2,3,\cdots,n\}$ is an orthonormal basis for $\ra(A)$, and the point (a) holds. Hence the lemma follows. 
\end{proof}
\begin{lemma}\label{Pfr}
In the lemma \ref{Lfr}, the set $\{u_i\,\mid\,i=1,2,3,\cdots,n\}$ of vectors can be chosen to be an \textit{orthonormal basis} for $\ra(A)$, and the adjoint of $A$ can be written as \begin{equation}\label{Adgr}
A^\dagger\psi=\sum_{i=1}^{n}v_i\langle u_i\mid \psi\rangle_{\ur},
\end{equation} for all $\psi\in\ur$. Furthermore, $A^\dagger u_i=v_i$ for all $i=1,2,3,\cdots,n$ and the set $\{v_i\,\mid\,i=1,2,3,\cdots,n\}$ of vectors can be chosen to be a \textit{basis} for $\ra(A^\dagger)$. 
\end{lemma}
\begin{proof}
In the lemma \ref{Lfr}, one can trivially observe that the set $\{u_i\,\mid\,i=1,2,3,\cdots,n\}$ of vectors can be chosen to be an orthonormal basis for $\ra(A)$. Now let $(\phi,\psi)\in\vr\times\ur$, and from equation \ref{fr}, we have 
$$\langle\psi\mid A\phi\rangle_{\ur}=\sum_{i=1}^{n}\langle\psi\mid u_i\rangle_{\ur}\langle v_i\mid\phi\rangle_{\vr}.$$
This implies that,
$\langle A^\dagger\psi\mid \phi\rangle_{\vr}=\langle \,\sum_{i=1}^{n}v_i\overline{\langle\psi\mid u_i\rangle}_{\ur}\mid\phi\rangle_{\vr}$ and so the equation \ref{Adgr} holds. From equation \ref{Adgr}, it immediately follows that  $A^\dagger u_i=v_i$ for all $i=1,2,3,\cdots,n$, as $\langle u_i\mid u_j\rangle_{\ur}=\delta_{ij}$. This guarantees that for each $i=1,2,3,\cdots,n$, $v_i\in\ra(A^\dagger)$, and it suffices, together with (\ref{Adgr}), to say that $\ra(A^\dagger)=\text{right span over}~\, \quat\{v_i\,\mid\,i=1,2,3,\cdots,n\}$. Now assume that $$\exists\,\pu_i\in\quat:~i=1,2,3,\cdots,n \mbox{~~such that~~}\sum_{i=1}^{n}v_i\pu_i=0.$$ 
Then $\sum_{i=1}^{n}(A^\dagger u_i)\pu_i=0$, and $\sum_{i=1}^{n}u_i\pu_i\in\kr(A^\dagger)=\ra(A)^\perp$ by (a) of proposition \ref{IP30}. Thus $\sum_{i=1}^{n}u_i\pu_i=0$ as $u_i\in\ra(A)$, for all $i=1,2,3,\cdots,n$. Since $\{u_i\,\mid\,i=1,2,3,\cdots,n\}$ forms a (orthonormal) basis, we have for each $i=1,2,3,\cdots,n$, $\pu_i=0$. That is, $\{v_i\,\mid\,i=1,2,3,\cdots,n\}$ is a linear independent set in $\ra(A^\dagger)$. This completes the proof.   
\end{proof}
\begin{theorem}\label{PD22} Let $A\in\B(\vr,\ur)$ be a finite rank operator. Then
\begin{itemize}
\item [(a)] $A$ is compact.
\item [(b)] $A^\dagger\in\B(\ur,\vr)$ is a finite rank operator and \em$\dim(\ra(A))=\dim(\ra(A^\dagger))$.
\end{itemize}
\end{theorem} 
\begin{proof}
Since the operator $A:\vr\longrightarrow\ur$ is a finite rank, we have $\dim(A(U))<\dim(\ra(A))<\infty$, where $U=\{\phi\in\vr~|~\|\phi\|_{\vr}\leq1\}$. Hence by proposition \ref{ort}, $A(U)$ is closed. Now for each $\phi\in U$,
$$\|A\phi\|_{\ur}\leq \|A\|\|\phi\|_{\vr}\leq\|A\|.$$
That is, $A(U)$ is closed and bounded set. Hence by proposition \ref{Com}, $\overline{A(U)}=A(U)$ is compact, and so $A$ is a compact operator. This concludes the point (a). 
From the lemmas \ref{Lfr} and \ref{Pfr}, the point (b) immediately follows.
\end{proof}
\begin{definition}\label{PD3}
Let $M\subset\vr$ be a closed subspace, then $\text{codim}(M)=\dim(\vr/M)$.
\end{definition}
\begin{definition}\label{PD4}
Let $A:\vr\longrightarrow\ur$ be a bounded operator, then $\text{coker}(A):=\ur/\text{ran}(A)$ and $\dim(\text{coker}(A))=\dim(\ur)-\dim(\text{ran}(A)).$
\end{definition}
\begin{theorem}\label{comeq}
If $A\in\B(\ur,\vr)$, then the following statements are equivalent:
\begin{itemize}
\item [(a)] $A$ is compact.
\item [(b)] There exist finite rank operators $A_n:\vr\longrightarrow\ur$ such that $\|A-A_n\|\longrightarrow 0$ as $n\longrightarrow \infty$.
\item [(c)] $A^\dagger$ is compact.
\end{itemize}
\end{theorem}
\begin{proof}
Suppose that point (a) holds. Then $\overline{A(U)}$ is a compact set, where $U=\{\phi\in\vr~|~\|\phi\|_{\vr}\leq1\}$. So $\overline{A(U)}$ is separable, and there exists an orthonormal basis $\{e_i~|~i=1,2,3,\cdots\}$ for $\overline{A(U)}$. Now, for any $n\in\mathbb{N}$, define $A_n:\vr\longrightarrow\ur$ by
$$A_n\phi=\sum_{i=1}^{n}e_i\langle A^\dagger e_i\mid \phi\rangle_{\vr},\mbox{~~for all $\phi\in\vr$.}$$
Then by lemma (\ref{Lfr}), $A_n$ is a finite rank operator, for all $n\in\mathbb{N}$. Further, for every $\phi\in U$, we have
$$\|A\phi-A_n\phi\|^2_{\ur}=\|\sum_{i>n}e_i\langle A^\dagger e_i\mid \phi\rangle_{\vr}\|^2_{\ur}=\|\sum_{i>n}e_i\langle  e_i\mid A\phi\rangle_{\ur}\|^2_{\ur}=\sum_{i>n}\mid\langle  e_i\mid A\phi\rangle_{\ur}\mid^2.$$ 
Thus $\|A\phi-A_n\phi\|_{\ur}\longrightarrow 0$ as $n\longrightarrow\infty$, for all $\phi\in U$, since for each $\phi\in U$, 
$$\sum_{i>n}\mid\langle  e_i\mid A\phi\rangle_{\ur}\mid^2=\left |\sum_{i=1}^{\infty}\mid\langle  e_i\mid A\phi\rangle_{\ur}\mid^2-\sum_{i=1}^{n}\mid\langle  e_i\mid A\phi\rangle_{\ur}\mid^2\right |\longrightarrow0$$ as $n\longrightarrow\infty$. Therefore,
$$\|A-A_n\|=\sup_{\|\phi \|_{\vr}=1}\|A\phi-A_n\phi \|_{\ur}\longrightarrow 0$$ as $n\longrightarrow\infty$. Hence point (b) holds.\\
Suppose that point (b) holds. Then $\|A^\dagger-A_n^\dagger\|=\|A-A_n\|\longrightarrow 0$ as $n\longrightarrow\infty$. Thus point (c) follows from theorem (\ref{PD22}) and (c) of proposition (\ref{com1}).\\
The implication (c) $\Rightarrow$ (a) can be obtained, by applying the implications (a) $\Rightarrow$ (b) and (b) $\Rightarrow$ (c) for $A^\dagger$ and using the fact that $A^{\dagger\dagger}=A$. Hence the theorem follows.
\end{proof}
\begin{proposition}\label{PP22}
If $A\in\B(\vr)$ is compact, then $\dim(\kr(\qu\Iop-A))<\infty$, for all $\qu\in\quat\smallsetminus\{0\}$.
\end{proposition}
\begin{proof}
Let $\qu\in\quat\smallsetminus\{0\}$. Assume that $\dim(\kr(\qu\Iop-A))=\infty$. Then there exists a sequence $\{\phi_n\}\subseteq\kr(\qu\Iop-A)$ such that \begin{center}
$\|\phi_n\|_{\vr}=1$ and $\|\phi_n-\phi_m\|_{\vr}>\dfrac{1}{2}$, for any $m,n\in\mathbb{N}$ with $m\ne n$,
\end{center}
as $\kr(\qu\Iop-A)$ is separable. Now for each $n\in\mathbb{N}$, $A\phi_n=\qu\phi_n$ implies 
\begin{center}
$\|A\phi_n\|_{\vr}=\mid\qu\mid$ and $\|A\phi_n-A\phi_m\|_{\vr}>\dfrac{\mid\qu\mid}{2}$, for any $m,n\in\mathbb{N}$ with $m\ne n$.
\end{center}
This suffices to say that $A$ is not compact. This concludes the result. 
\end{proof}

\section{Approximate $S$-point spectrum}
Following the complex case, we define the approximate spherical point spectrum of $A\in\B(V_\quat^R)$ as follows.
\begin{definition}\label{D1}
Let $A\in\B(V_\quat^R)$. The {\em approximate S-point spectrum} of $A$, denoted by $\apo(A)$, is defined as
$$\apo(A)=\{\qu\in\quat~~|~~\text{there is a sequence}~~\{\phi_n\}_{n=1}^{\infty}~~\text{such that}~~\|\phi_n\|=1~~\text{and}~~\|R_\qu(A)\phi_n\|\longrightarrow 0\}.$$
\end{definition}
\begin{proposition}\label{P3}
Let $A\in\B(\vr)$, then $\sigma_{pS}(A)\subseteq\apo(A)$. 
\end{proposition}
\begin{proof}
Let $\qu\in\sigma_{pS}(A)$, then $\text{ker}(R_\qu(A))\not=\{0\}$. That is $\{\phi\in\D(A^2)~|~R_\qu(A)\phi=0\}\not=\{0\}.$ Therefore, there exists $\phi\in\D(A^2)$ and $\phi\not=0$ such that $R_\qu(A)\phi=0$. Set $\phi_n=\phi/\|\phi\|$ for all $n=1,2,\cdots$. Then $\|\phi_n\|=1$ for all $n$ and $\|R_\qu(A)\phi_n\|\longrightarrow 0$ as $n\longrightarrow\infty$. Thus $\qu\in\apo(A)$.
\end{proof}
\begin{proposition}\label{P4}
If $A\in\B(\vr)$ and $\qu\in\quat$, then the following statements are equivalent.
\begin{enumerate}
\item[(a)] $\qu\not\in\apo(A).$
\item[(b)] $\text{ker}(R_\qu(A))=\{0\}$ and $\text{ran}(R_\qu(A))$ is closed.
\item[(c)] There exists a constant $c\in\R$, $c>0$ such that $\|R_\qu(A)\phi\|\geq c\|\phi\|$ for all $\phi\in\D(A^2)$.
\end{enumerate}
\end{proposition}
\begin{proof}
(a)$\Rightarrow$ (c): Suppose (c) fails to hold. Then for every $n$, there exists non-zero vectors $\phi_n$ such that
$$\|R_\qu(A)\phi_n\|\leq\frac{\|\phi_n\|}{n}.$$
For each $n$, let $\displaystyle\psi_n=\frac{\phi_n}{\|\phi_n\|}$, then $\|\psi_n\|=1$ and 
$\displaystyle\|R_\qu(A)\psi_n\|<\frac{1}{n}\longrightarrow 0$ as $n\longrightarrow \infty.$ Therefore $\q\in\apo(A).$\\
(c)$\Rightarrow$ (b): Suppose that there is a constant $c>0$ such that $\|R_\qu(A)\phi\|\geq c\|\phi\|$ for all $\phi\in\D(A^2)$. Therefore, $R_\qu(A)\phi\not=0$ for all $\phi\not=0$ and $R_\qu(A)0=0$. Therefore $\text{ker}(R_\qu(A))=\{0\}$.
Let $\phi\in\overline{\text{ran}(R_\qu(A))}$, where the bar stands for the closure. Then there exists a sequence $\{\phi_n\}\subseteq\text{ran}(R_\qu(A))$ such that $\phi_n\longrightarrow\phi$ as $n\longrightarrow\infty$. From (c) we have a constant $c>0$ such that $\displaystyle \|\phi_n-\phi_m\|\leq\frac{1}{c}\|R_\qu(A)\phi_n-R_\qu(A)\phi_m\|.$ Therefore $\{\phi_n\}$ is a Cauchy sequence as $\{R_\qu(A)\phi_n\}$ is a Cauchy sequence. Let $\displaystyle\psi=\lim_{n\rightarrow\infty}\phi_n$, then $\displaystyle R_\qu(A)\psi=\lim_{n\rightarrow\infty} R_\qu(A)\phi_n$ as $R_\qu$ is continuous. Therefore $R_\qu(A)\psi=\phi$ and hence $\phi\in\text{ran}(R_\qu(A)),$ from which we get $\overline{\text{ran}(R_\qu(A))}\subseteq \text{ran}(R_\qu(A))$. Hence $\text{ran}(R_\qu(A))$ is closed.\\
(b)$\Rightarrow$(a): Suppose that $\text{ker}(R_\qu(A))=\{0\}$ and $\text{ran}(R_\qu(A))$ is closed. Since $A\in\B(\vr)$, $R_\qu(A)$ is bounded and $R_\qu(A):\D(A^2)\longrightarrow \text{ran}(R_\qu(A))$ is a bijection. Therefore, by theorem \ref{open}, there is a bounded linear operator $B:\text{ran}(R_\qu(A))\longrightarrow \D(A^2)$ such that $BR_\qu(A)\phi=\phi$ for all $\phi\in\D(A^2)$. Thus,
$\|\phi\|=\|BR_\qu(A)\phi\|\leq \|B\|\|R_\qu(A)\phi\|$ for all $\phi\in\D(A^2).$ Therefore, if there is a sequence $\{\phi_n\}\subseteq\D(A^2)$ with $\|\phi_n\|=1$, then
$1\leq  \|B\|\|R_\qu(A)\phi_n\|$. That is, $\displaystyle\frac{1}{\|B\|}\leq\|R_\qu(A)\phi_n\|.$ Hence, there is no sequence $\{\phi_n\}\subseteq\D(A^2)$ with $\|\phi_n\|=1$ such that $\|R_\qu(A)\phi_n\|\longrightarrow 0$ as $n\longrightarrow\infty.$ Therefore $\qu\not\in\apo(A)$.
\end{proof}
\begin{theorem}\label{T1}Let $A\in\B(\vr)$, then $\apo(A)$ is a non-empty closed subset of $\quat$ and $\partial\sigma_S(A)\subseteq\apo(A),$ where $\partial\sigma_S(A)$ is the boundary of $\sigma_S(A)$.
\end{theorem}
\begin{proof}
Let $\qu\in\partial\sigma_S(A).$ Since $\rho_S(A)\not=\emptyset$ and $\sigma_S(A)$ is closed,
$$\partial\sigma_S(A)=\partial\rho_S(A)=\overline{\rho_S(A)}\cap\sigma_S(A).$$
Therefore, there exists a sequence $\{\qu_n\}\subseteq\rho_S(A)$ such that $\qu_n\longrightarrow\qu$ as $n\longrightarrow \infty$. Since,
$$R_{\qu_n}(A)-R_\qu(A)=(\|\qu_n\|-\|\qu\|)\Iop-2(\text{Re}(\qu_n)-\text{Re}(\qu))A,\quad\text{for all}~~n,$$
we have $R_{\qu_n}(A)\longrightarrow R_\qu(A)$ in $\B(\vr)$ as $\|\qu_n\|\longrightarrow\|\qu\|$ and $\text{Re}(\qu_n)\longrightarrow\text{Re}(\qu)$.\\
{\bf Claim:} $\displaystyle\sup_n\|R_{\qu_n}(A)^{-1}\|=\infty.$\\
Since $\sigma_S(A)$ is closed, $\q_n\in\rho_S(A)$, $\qu\in\partial\sigma_S(A)$ and $\qu\not\in\rho_S(A)$ we have $R_{\qu_n}(A)\longrightarrow R_{\qu}(A)$ in $\B(\vr)$ and $R_\qu(A)\in\B(\vr)\setminus\mathcal{G}(\vr).$\\
{\em Case-I:} $\qu\in\apo(A)$. Assume $\displaystyle\sup_n\|R_{\qu_n}(A)^{-1}\|<\infty.$ Since $\qu\in\apo(A)$, $\text{ker}(R_\qu(A))\not=\{0\}.$ Therefore. there exits $\phi\in\D(A^2)$ with $\phi\not=0$ such that $R_\qu(A)\phi=0$. Thus,
\begin{eqnarray*}
0\not=\|\phi\|&=& \|R_{\qu_n}(A)^{-1}R_{\qu_n}(A)\phi\|\\
&\leq&\sup_n\|R_{\qu_n}(A)^{-1}\|\limsup_n\|R_{\qu_n}(A)\phi\|\\
&\leq& \sup_n\|R_{\qu_n}(A)^{-1}\| \|R_{\qu}(A)\phi\|=0,
\end{eqnarray*}
which is a contradiction. Therefore  $\displaystyle\sup_n\|R_{\qu_n}(A)^{-1}\|=\infty.$\\
{\em Case-II:} $\qu\not\in\apo(A)$. Then $\text{ker}(R_\qu(A))=\{0\}$. Thus $R_\qu(A)^{-1}:\text{ran}(R_\qu(A))\longrightarrow\vr$ exists. For each $n$, we have
$$R_{\qu_n}(A)^{-1}-R_\qu(A)^{-1}=R_{\qu_n}(A)^{-1}[R_\qu(A)-R_{\qu_n}(A)]R_\qu(A)^{-1}.$$
Assume $\displaystyle\sup_n\|R_{\qu_n}(A)^{-1}\|<\infty.$ Then, since $R_{\qu_n}(A)\longrightarrow R_\qu(A)$ in $\B(\vr)$, $R_{\qu_n}(A)^{-1}\longrightarrow R_\qu(A)^{-1}$ and $R_{\qu}(A)^{-1}\in\B(\vr)$, that is, $R_\qu(A)\in\mathcal{G}(\vr)$, which is a contradiction to $\qu\in\sigma_S(A).$ Therefore, $\displaystyle\sup_n\|R_{\qu_n}(A)^{-1}\|=\infty.$ The claim is proved.\\
Since $\displaystyle\|R_{\qu_n}(A)^{-1}\|=\sup_{\|\psi\|=1}\|R_{\qu_n}(A)^{-1}\psi\|$, we can take vectors $\phi_n\in\vr$ with $\|\phi_n\|=1$, such that, for each $n$,
$$\displaystyle\|R_{\qu_n}(A)^{-1}\|-\frac{1}{n}\leq \displaystyle\|R_{\qu_n}(A)^{-1}\phi_n\|\leq \displaystyle\|R_{\qu_n}(A)^{-1}\|.$$
From which we have $\displaystyle\sup_n\|R_{\qu_n}(A)^{-1}\phi_n\|=\infty$, and hence
$\displaystyle\inf_n\|R_{\qu_n}(A)^{-1}\phi_n\|^{-1}=0$.
Therefore, there are subsequences of $\{\phi_n\}$ and $\{\qu_n\}$, call them $\{\phi_k\}$ and $\{\qu_k\}$ such that
\begin{equation}\label{TE1}
\|R_{\qu_k}(A)^{-1}\phi_k\|^{-1}\longrightarrow 0,\quad\text{as}~~k\longrightarrow \infty.
\end{equation}
Set $\displaystyle \psi_k=\frac{R_{\qu_k}(A)^{-1}\phi_k}{\|R_{\qu_k}(A)^{-1}\phi_k\|}$, then $\{\psi_k\}$ is a sequence of unit vectors in $\vr$ with
$$\|R_{\qu_k}(A)\psi_k\|=\frac{\|R_{\qu_k}(A)R_{\qu_k}(A)^{-1}\phi_k\|}{\|R_{\qu_k}(A)^{-1}\phi_k\|}=\|R_{\qu_k}(A)^{-1}\phi_k\|^{-1}.$$
Now, from (\ref{TE1}) and as $\qu_k\longrightarrow\qu$, we have
\begin{eqnarray*}
\|R_\qu(A)\psi_k\|&=&\|(A^2-2\text{Re}(\qu)A+|\qu|^2\Iop)\psi_k\|\\
&=&\|(A^2-2\text{Re}(\qu_k)A+|\qu_k|^2\Iop+2\text{Re}(\qu_k)A-|\qu_k|^2\Iop-2\text{Re}(\qu)A+|\qu|^2\Iop)\psi_k\|\\
&=&\|(R_{\qu_k}(A)-(|\qu_k|^2-|\qu|^2)\Iop+2(\text{Re}(\qu_k)-\text{Re}(\qu))A)\psi_k\|\\
&\leq& \|(R_{\qu_k}(A)\psi_k\|+||\qu_k|^2-|\qu|^2|+2|\text{Re}(\qu_k)-\text{Re}(\qu)|~\|A\|\\
&=&
\|R_{\qu_k}(A)^{-1}\phi_k\|^{-1}+||\qu_k|^2-|\qu|^2|+2|\text{Re}(\qu_k)-\text{Re}(\qu)|~\|A\|
\longrightarrow 0\quad\text{as}~~~k\longrightarrow\infty.
\end{eqnarray*}
That is, $\|R_\qu(A)\psi_k\|\longrightarrow 0$ as $k\longrightarrow\infty.$ Hence $\qu\in\apo(A)$ and therefore $\partial\sigma_S(A)\subseteq\apo(A)$. Further, from remark (\ref{R1}) $\partial\sigma_S(A)\not=\emptyset$. Thus $\apo(A)\not=\emptyset$.\\
Finally, take an arbitrary $\qu\in\quat\setminus\apo(A)$ such that $R_\qu(A)$ is bounded below. Then, for all $\phi\in\D(A^2)$ and $\pu\in\quat$, there exists $\alpha>0$ $(\alpha\in\R)$ such that
\begin{eqnarray*}
\alpha\|\phi\|&\leq& \|R_\qu(A)\phi\|\\
&\leq& \|R_\pu(A)\phi\|+\|(|\pu|^2-|\qu|^2)\phi\|+2\|(\text{Re}(\pu)-\text{Re}(\qu))A\phi\|,
\end{eqnarray*}
which implies
$$(\alpha-|~|\pu|^2-|\qu|^2~|-2|(\text{Re}(\pu)-\text{Re}(\qu))|\|A\|~)\|\phi\|\leq\|R_\pu(A)\phi\|.$$
Hence $R_\pu(A)$ is bounded below for each $\pu\in\quat$ that satisfies $$\alpha>|~|\pu|^2-|\qu|^2~|+2|(\text{Re}(\pu)-\text{Re}(\qu))|\|A\|.$$
From this we can say that $\pu\in\quat\setminus\apo(A)$ for every $\pu$ which is sufficiently close to $\qu$. Hence, the open ball centered at $\qu$ with radius $\alpha$, $B_{\alpha}(\qu)\subseteq\quat\setminus\apo(A)$. Therefore $\quat\setminus\apo(A)$ is open and $\apo(A)$ is closed.
\end{proof}
\begin{theorem}\label{T2}
Let $A\in\B(\vr)$ and $\qu\in\quat$, then the following statements are equivalent.
\begin{enumerate}
\item [(a)]$\qu\not\in\apo(A)$.
\item[(b)]$\qu\not\in\sigma_l^S(A).$
\item[(c)]$\oqu\not\in\sigma_r^S(A^\dagger).$
\item[(d)]$\text{ran}(R_{\oqu}(A^\dagger))=\vr.$
\end{enumerate}
\end{theorem}
\begin{proof}
For $B\in\B(\vr)$, we have: $BR_{\qu}(A)=\Iop$ if and only if $R_{\oqu}(A^\dagger)B^\dagger=\Iop.$ Therefore (b) and (c) are equivalent.\\
(a)$\Rightarrow$ (b): Suppose $\qu\not\in\apo(A)$. Then by proposition (\ref{P4}), $\text{ker}(\R_\qu(A))=\{0\}$. Hence $R_\qu(A)$ is invertible, and therefore $R_\qu(A)$ is left invertible. Thus $\qu\not\in\sigma_l^S(A).$\\
(c)$\Rightarrow$ (d): Since $\oqu\not\in\sigma_r^S(A^\dagger)$, $R_{\oqu}(A^\dagger)$ is right invertible. Therefore, there is an operator $C\in\B(\vr)$ such that $R_{\oqu}(A^\dagger)C=\Iop$. Hence,
$\vr=R_{\oqu}(A^\dagger)C(\vr)\subseteq\text{ran}(R_{\oqu}(A^\dagger))$. Therefore $\vr=\text{ran}(R_{\oqu}(A^\dagger))$.\\
(d)$\Rightarrow$ (a): Suppose $\text{ran}(R_{\oqu}(A^\dagger))=\vr.$
Define $T:\text{ker}(R_{\oqu}(A^\dagger))^{\perp}\longrightarrow\vr$ by $T\phi=R_{\oqu}(A^\dagger)\phi$. Then $\text{ker}(T)=\{0\}$, that is, $T$ is bijective and hence invertible. Let $$C:\vr\longrightarrow\vr\quad\text{ by}\quad C\phi=T^{-1}\phi.$$ Then $C\vr=\text{ker}(R_{\oqu}(A^\dagger))^{\perp}$ and $R_{\oqu}(A^\dagger)C=\Iop$. Hence, $C^{\dagger}R_\qu(A)=\Iop$ and, for $\phi\in\vr$,
$$\|\phi\|=\|C^{\dagger}R_\qu(A)\phi\|\leq \|C^\dagger\|~\|R_\qu(A)\phi\|,$$ from this we get $\displaystyle\inf\{\|R_\qu(A)\phi\|~|~\|\phi\|=1\}\geq \|C^\dagger\|^{-1}.$ Therefore, there is no sequence $\{\phi_n\}\subseteq\vr$ with $\|\phi_n\|=1$ and $\|R_\qu(A)\phi_n\|\longrightarrow 0$. Hence $\qu\not\in\apo(A)$.
\end{proof}
\begin{corollary}\label{C1}
If $A\in\B(\vr)$, then $\partial\sigma_S(A)\subseteq\sigma_l^S(A)\cap\sigma_r^S(A)=\apo(A)\cap\apo(A^\dagger)^*$.
\end{corollary}
\begin{proof}
From theorem \ref{T2}, we have $\qu\not\in\apo(A)\Leftrightarrow\qu\not\in\sigma_l^S(A)$, and hence $\apo(A)=\sigma_l^S(A).$ Also $q\not\in\apo(A^\dagger)\Leftrightarrow\oqu\not\in\sigma_r^S(A)\Leftrightarrow \qu\not\in\sigma_r^S(A)^*$, and therefore $\apo(A^\dagger)^*=\sigma_r^S(A).$ Let $\qu\in\partial\sigma_S(A)$, then by theorem \ref{T1} $\qu\in\apo(A)=\sigma_l^S(A)$.
If $\oqu\not\in\apo(A^\dagger)$, then by theorem \ref{T2} $\text{ran}(R_{\oqu}(A^\dagger))=\vr$ and by proposition \ref{P4} $\text{ker}(R_{\oqu}(A^\dagger))=\{0\}$. Thus $R_{\oqu}(A^\dagger)$ is invertible. Therefore $R_\qu(A)$ is invertible as $A\in\B(\vr)$, and hence $\qu\not\in\sigma_S(A)$, which is a contradiction. Therefore, $\oqu\in\apo(A^\dagger)$, which implies $\qu\in\apo(A^\dagger)^*=\sigma_r^S(A)$. Thus $\qu\in\sigma_l^S(A)\cap\sigma_r^S(A)$.   
\end{proof}
Following the complex formalism in the following we define the S-compression spectrum for an operator $A\in\B(\vr)$.
\begin{definition}\label{DC}
The spherical compression spectrum of an operator $A\in\B(\vr)$, denoted by $\sigma_c^S(A)$, is defined as
$$\sigma_c^S(A)=\{\qu\in\quat~~|~~\text{ran}(R_\qu(A))~~~\text{is not dense in}~~\vr~\}.$$
\end{definition}
\begin{proposition}\label{CP1}
Let $A\in\B(\vr)$ and $\qu\in\quat$. Then,
\begin{enumerate}
\item[(a)] $\qu\in\sigma_c^S(A)$ if and only if $\oqu\in\sigma_{pS}(A)$.
\item[(b)] $\sigma_c^S(A)\subseteq\sigma_r^S(A)$.
\item[(c)] $\sigma_S(A)=\apo(A)\cup\sigma_c^S(A)$.
\end{enumerate}
\end{proposition}
\begin{proof}
(a)~$\qu\in\sigma_c^S(A)\Leftrightarrow \text{ran}(R_\qu(A))=\text{ker}(R_{\oqu}(A^\dagger))^\perp $ is not dense in $\vr\Leftrightarrow \text{ker}(R_{\oqu}(A^\dagger))\not=\{0\}\Leftrightarrow \oqu\in\sigma_{pS}(A^\dagger).$\\
(b)~Let $\qu\in \sigma_c^S(A)$, then $\text{ran}(R_\qu(A))$ is not dense in $\vr$. If $\qu\not\in\sigma_r^S(A)$, then by theorem \ref{T2} $\oqu\not\in\apo(A^\dagger)$, thus, again by theorem \ref{T2}, $\text{ran}(R\qu(A))=\vr$, which is contradiction. Therefore $\qu\in\sigma_r^S(A)$ and hence $\sigma_c^S(A)\subseteq\sigma_r^S(A).$\\
(c)~Let $\qu\in\sigma_S(A)$, then $R\qu(A)$ is not invertible in $\B(\vr)$. That is, $\text{ker}(R_\qu(A))\not=\{0\}$ or $\text{ran}(R_\qu(A))\not=\vr$. Therefore, by proposition \ref{P4}, $\qu\in\apo(A)$. Hence $\sigma_S(A)\subseteq\apo(A)\cup\sigma_c^S(A)$.\\
Let $\qu\in\apo(A)\cup\sigma_c^S(A)$. If $\qu\in\apo(A)$, then by proposition \ref{P4} $\text{ker}(R_\qu(A))\not=\{0\}$ or $\text{ran}(R_\qu(A))$ is not closed. Hence,  $\text{ker}(R_\qu(A))\not=\{0\}$ or $\text{ran}(R_\qu(A))\not=\vr$. Thus $\qu\in\sigma_S(A)$. If $\qu\in\sigma_c^S(A)$, then  $\text{ran}(R_\qu(A))$ is not dense in $\vr$. Therefore  $\text{ran}(R_\qu(A))\not=\vr$ and hence $\qu\in\sigma_S(A)$. In summary $\apo(A)\cup\sigma_c^S(A)\subseteq \sigma_S(A)$. Hence (c) holds.
\end{proof}
\section{Fredholm operators in the quaternionic setting}
In the complex setting Fredholm operators are studied in Banach spaces and Hilbert spaces for bounded and even to unbounded linear operators. In this section we shall study the theory for quaternionic bounded linear operators on separable Hilbert spaces. In this regard let $\vr$ and $\ur$ be two separable right quaternionic Hilbert spaces.
\begin{definition}\label{FD1}
A Fredholm operator is an operator $A\in\B(\vr,\ur)$ such that $\text{ker}(A)$ and $\text{coker}(A)=\ur/\text{ran}(A)$ are finite dimensional. The dimension of the cokernel is called the codimension, and it is denoted by $\text{codim}(A).$
\end{definition}
\begin{proposition}\label{FP1}
If $A\in\B(\vr,\ur)$ is a Fredholm operator, then $\text{ran}(A)$ is closed.
\end{proposition}
\begin{proof}
Restrict $A$ to $\text{ker}(A)^\perp$, that is, $A\vert_{\text{ker}(A)^\perp}:\text{ker}(A)^\perp\longrightarrow\ur$. Then we can assume that $\text{ker}(A)=\{0\}.$ Since $A$ is Fredholm, assume that $\text{codim}(A)=\dim(\ur/\text{ran}(A))<\infty.$ Therefore, there exist a finite dimensional subspace $V\subseteq\ur$ such that $\ur=\text{ran}(A)\oplus V$. Since $V$ is finite dimensional, $V$ is closed. Hence $\ur=V\oplus V^\perp$. Let $P:\ur\longrightarrow V^\perp$ be the orthogonal projection onto $V^\perp$. Hence $P$ is a bounded bijection. Let $G:=PA:\text{ker}(A)^\perp\longrightarrow V^\perp$, which is a composition of two bounded operators and so $G$ is a bounded bijection as $A$ and $P$ are bounded bijections and thereby continuous. Therefore by the open mapping theorem $G^{-1}: V^\perp\longrightarrow \text{ker}(A)^\perp$ is bounded.\\
Let $\psi\in\overline{\text{ran}(A)}$. Then, there exist a sequence $\{\phi_n\}\subset\vr$ such that $\displaystyle\lim_n A(\phi_n)=\psi$. Thus $\displaystyle\lim_n PA(\phi_n)=P(\psi)$ exists in $V^\perp$. That is $\displaystyle\lim_n G(\phi_n)=P(\psi)$ exists in $V^\perp$. Hence, $\phi:=\displaystyle\lim_n \phi_n=G^{-1}P(\psi)=A^{-1}(\psi)$ exists in $\vr$. Therefore, $\psi=A(\phi)\in\text{ran}(A)$. Hence $\text{ran}(A)$ is closed. 
\end{proof}
\begin{definition}\label{FD2}
Let $A\in\B(\vr,\ur)$ be a Fredholm operator. Then the index of $A$ is the integer, $\text{ind}(A)=\dim(\text{ker}(A))-\dim(\text{coker}(A)).$
\end{definition}
\begin{remark}\label{FR1}
Since $\ra(A)$ is closed, we have $\ur=\ra(A)\oplus\ra(A)^\perp=\ra(A)\oplus\kr(A^\dagger).$
Therefore, $\text{coker}(A)=\ur/\ra(A)\cong\kr(A^\dagger).$ Thus,
$$\text{ind}(A)=\dim(\kr(A))-\dim(\kr(A^\dagger)).$$
\end{remark}
\begin{theorem}\label{FT1}
Let $A\in\B(\vr,\ur)$ be bijective, and let $K\in\B_0(\vr,\ur)$ be compact. Then $A+K$ is a Fredholm operator.
\end{theorem}
\begin{proof}
Since $A$ is bijective, $A^{-1}$ exists. Since the kernel of a bounded right linear operator is closed, $\kr(A+K)$ is a Hilbert space. Thus, for $\phi\in\kr(A+K)$, we have $A\phi=-K\phi$. Let $\{\phi_n\}_{n=1}^{\infty}\subseteq\kr(A+K)$ be a bounded sequence. Since $K$ is a compact operator, $\{K\phi_n\}_{n=1}^{\infty}$ has a convergence subsequence, $\{K\phi_{n_k}\}_{k=1}^{\infty}$. Since, for all $k$, $ \phi_{n_{k}}\in\kr(A+K)$, we have
$$\{K\phi_{n_{k}}\}_{k=1}^\infty=\{-A\phi_{n_k}\}_{k=1}^\infty.$$
Therefore, $\displaystyle\lim_{k\rightarrow\infty}A\phi_{n_k}=\psi$ exists. Hence, as $A^{-1}$ is continuous, we get $\displaystyle\lim_{k\longrightarrow\infty}\phi_{n_k}=A^{-1}\psi.$ That is, any bounded sequence in $\kr(A+K)$ has a convergence subsequence. Since an infinite dimensional Hilbert space has an infinite orthonormal sequence with no convergence subsequence, $\dim(\kr(A+K))<\infty.$\\
Since $A$ is invertible and $K$ is compact, $A^\dagger$ is invertible and $K^\dagger$ is compact. Therefore, by the above argument $\dim(\kr(A^\dagger+K^\dagger))<\infty.$ Thus, since $\ur=\overline{\ra(A+K)}\oplus\kr(A^\dagger+K^\dagger)$, to show that $\text{codim}(A+K)<\infty$, it is enough to show that $\ra(A+K)$ is closed. Split $\vr$ as $\vr=\HI\oplus\kr(A+K)$, where $\HI$ is a subspace of $\vr$ such that $\HI\perp\kr(A+K)$, and consider the restriction of $A+K$ to $\HI$.\\
{\em Claim:} For $\phi\in\HI$, the inequality
$$\|\phi\|_{\vr}\leq c\|(A+K)\phi\|_{\ur}$$
holds for some $c>0$.\\
If for all $c>0$ there exists $\phi\in\HI$ such that $\|\phi\|_{\vr}> c\|(A+K)\phi\|_{\ur},$ then there exists sequences $\{c_n\}_{n=1}^{\infty}\subseteq (0,\infty)$ and $\{\phi_n\}_{n=1}^\infty\subseteq\HI$ with $\|\phi_n\|_{\vr}=1$ for all $n$ such that $c_n\longrightarrow\infty$ as $n\longrightarrow\infty$. Further, for all $n$,
$$1=\|\phi_n\|_{\vr}> c_n||(A+K)\phi_n\|_{\ur}.$$
Thus $\displaystyle\|(A+K)\phi_n\|_{\ur}<\frac{1}{c_n}\longrightarrow 0$ as $n\longrightarrow\infty$. Since $K$ is compact and, for each $n$, $\|\phi_n\|_{\vr}=1$, there exists a convergent subsequence $\{K\phi_{n_k}\}_{k=1}^{\infty}$ of $\{K\phi_n\}_{n=1}^\infty$. Assume $K\phi_{n_k}\longrightarrow\psi\in\ur$ as $k\longrightarrow\infty.$ Then, $A\phi_{n_k}\longrightarrow -\psi\in\ur$ as $k\longrightarrow\infty$. Hence $\phi_{n_k}=A^{-1}A\phi_{n_k}\longrightarrow \eta=-A^{-1}\psi$ as $k\longrightarrow\infty$, where $\eta\in\HI$ and $\|\eta\|_{\vr}=1$ as $\|\phi_{n_k}\|_{\vr}=1$ for all $k$. But $\displaystyle(A+K)\eta=\lim_{k\longrightarrow\infty}(A\phi_{n_k}+K\phi_{n_k})=\lim_{k\longrightarrow\infty} 0=0$. Therefore $\eta\in\kr(A+K)$, which is a contradiction to $\HI\perp\kr(A+K)$. Hence the claim holds. Thus, by proposition \ref{IP4}, $\ra(A+K)$ is closed. Hence $\text{codim}(A+K)<\infty$ and $A+K$ is Fredholm.
\end{proof}
\begin{proposition}\label{FP2}
If $A\in\B(\vr,\ur)$ is Fredholm then $A^\dagger\in\B(\ur,\vr)$ is Fredholm.
\end{proposition}
\begin{proof}
Since $A$ is Fredholm, $\dim(\kr(A))<\infty$ and $\dim(\ur/\ra(A))<\infty$. Then by proposition \ref{FP1}, $\ra(A)$ is closed. Therefore, $\ur=\ra(A)\oplus\ra(A)^\perp=\ra(A)\oplus\kr(A^\dagger)$. Thus, $\ur/\ra(A)\cong\kr(A^\dagger)$. Hence $\dim(\kr(A^\dagger))<\infty$. Also, since by proposition \ref{IP30}, $\kr(A)$ is closed, we have $\text{coker}(A^\dagger)=\vr/\ra(A^\dagger)=\vr/\kr(A)^\perp\cong \kr(A)$
Thus $\dim(\text{coker}(A^\dagger))=\dim(\kr(A))<\infty.$ Hence $A^\dagger$ is Fredholm.
\end{proof}
\begin{remark}\label{FR2}
From propositions \ref{FP1} and \ref{FP2}, $\ra(A^\dagger)$ is closed. Therefore,
$$\text{ind}(A^\dagger)=\dim(\kr(A^\dagger))-\dim(\kr((A^\dagger)^\dagger))=\dim(\kr(A^\dagger))-\dim(\kr(A))=-\text{ind}(A).$$
\end{remark}
\begin{theorem}\label{FT3}
$A\in\B(\vr,\ur)$ is Fredholm if and only if there exist $S_1,S_2\in\B(\ur,\vr)$ and compact operators $K_1$ and $K_2$, on $\vr$ and $\ur$ respectively, such that
$$S_1A=\mathbb{I}_{\vr}+K_1\quad\text{and}\quad AS_2=\mathbb{I}_{\ur}+K_2.$$
\end{theorem}
\begin{proof}
($\Rightarrow$) Suppose $A\in\B(\vr,\ur)$ is a Fredholm operator. Then $A$ defines a bijective operator
$$\hat{A}:\HI_1=(\kr(A))^\perp\longrightarrow\HI_2=\ra(A)=(\kr(A^\dagger))^\perp.$$
Then $\hat{A}^{-1}:\HI_2\longrightarrow\HI_1$ exists and bounded. Consider the orthogonal projection operator $P_r:\ur\longrightarrow\HI_2$ and the inclusion operator $\mathfrak{i}:\HI_1\longrightarrow\vr$. Let $S_2=\mathfrak{i}\hat{A}^{-1}P_r:\ur\longrightarrow\vr$, then $S_2\in\B(\ur,\vr)$ and $AS_2=A\mathfrak{i}\hat{A}^{-1}P_r=\mathbb{I}_{\ur}-P_a$, where $P_a$ is the orthogonal projection operator $P_a:\ur\longrightarrow\kr(A^\dagger)$. Let $K_2=-P_a$. Since $\kr(A^\dagger)$ is finite dimensional, $K_2$ is a finite rank operator so it is compact. That is, $AS_2=\mathbb{I}_{\ur}+K_2$.\\
From proposition \ref{FP2}, $A^\dagger$ is Fredholm. Therefore, from the above argument, there exists $S_3\in\B(\vr,\ur)$ and a compact operator on $\vr$ such that $A^\dagger S_3=\mathbb{I}_{\vr}+K_3$. Hence,
$S_3^\dagger A=\mathbb{I}_{\vr}+K_3^\dagger$. Set $S_1=S_3^\dagger$ and $K_1=K_3^\dagger$, then $S_1\in\B(\ur,\vr)$,  $K_1$ is compact on $\vr$ and $S_1A=\mathbb{I}_{\vr}+K_1$.\\
($\Leftarrow$) Suppose that there are operators $S_1,S_2\in\B(\ur,\vr)$ and compact operators $K_1$ on $\vr$ and $K_2$ on $\ur$ such that $S_1A=\mathbb{I}_{\vr}+K_1$ and $AS_2=\mathbb{\ur}+K_2$. We have the obvious inclusions,
\begin{eqnarray}
\kr(A)&\subseteq&\kr(S_1A)=\kr(\mathbb{I}_{\vr}+K_1)\label{FE1}\\
\ra(A)&\supseteq&\ra(AS_2)=\ra(\mathbb{I}_{\ur}+K_2)\label{FE2}.
\end{eqnarray}
By theorem \ref{FT1}, $\mathbb{I}_{\vr}+K_1$ and $\mathbb{I}_{\ur}+K_2$ are Fredholm operators. Therefore from equation \ref{FE1}, $\dim(\kr(A))\leq\dim(\mathbb{I}_{\vr}+K_1)<\infty$ and by equation \ref{FE2}, $\text{codim}(A)\leq\text{codim}(\mathbb{I}_{\ur}+K_2)<\infty.$ Hence $A$ is a Fredholm operator.
\end{proof}
\begin{remark}\label{FR4}
Let $A\in\B(\vr,\ur),$ then
\begin{enumerate}
\item [(a)] $A$ is said to be left semi-Fredholm if there exists $B\in\B(\ur,\vr)$ and a compact operator $K_1$ on $\vr$ such that $BA=\Iop+K_1$. The set of all left semi-Fredholm operators are denoted by $\mathcal{F}_l(\vr,\ur)$ \cite{con}.
\item [(b)] $A$ is said to be right semi-Fredholm if there exists $B\in\B(\ur,\vr)$ and a compact operator $K_2$ on $\ur$ such that $AB=\Iopu+K_2$. The set of all right semi-Fredholm operators are denoted by $\mathcal{F}_r(\vr,\ur)$ \cite{con}.
\item[(c)] By theorem \ref{FT3}, the set of all Fredholm operators, $\mathcal{F}(\vr,\ur)=\mathcal{F}_l(\vr,\ur)\cap \mathcal{F}_r(\vr,\ur)$.
\item[(d)] From theorem \ref{FT3} it is also clear that every invertible right linear operator is Fredholm.
\item[(e)]Let $\mathcal{S}\FF(\vr)=\FF_l(\vr)\cup\FF_r(\vr)$. From theorem \ref{FT3} and theorem \ref{comeq}, we have
\begin{eqnarray*}
A\in\FF_l(\vr)&\Leftrightarrow& A^\dagger\in\FF_r(\vr)\\
A\in\mathcal{S}\FF(\vr)&\Leftrightarrow& A^\dagger\in\mathcal{S}\FF(\vr)\\
A\in\FF(\vr)&\Leftrightarrow& A^\dagger\in\FF(\vr).
\end{eqnarray*}
\end{enumerate}
\end{remark}
\begin{proposition}\label{FP3}
$A\in\B(\vr,\ur)$ is Fredholm if and only if  there exists $S\in\B(\ur,\vr)$ such that $AS-\mathbb{I}_{\ur}$ and $SA-\Iop$ are both finite rank operators.
\end{proposition}
\begin{proof}
($\Rightarrow$) Suppose $A$ is a Fredholm operator, then $\ra(A)$ is closed and $A:\kr(A)^\perp\longrightarrow\ra(F)$ is a bijection between Hilbert spaces. Let $\hat{A}$ be the inverse of this bijection. By the open mapping theorem, $\hat{A}\in\B(\ur,\vr)$. Let $P:\ur\longrightarrow\ra(A)$ be the orthogonal projection onto $\ra(A)$. Set $S=\hat{A}P$, then
$$SA-\Iop=\hat{A}PA-\Iop=\hat{A}A-\Iop=-P_k,$$
where $P_k$ is the orthogonal projection onto $\kr(A)$, hence a finite rank operator. Also
$$AS-\mathbb{I}_{\ur}=A\hat{A}P-\mathbb{I}_{\ur}=-(\mathbb{I}_{\ur}-P),$$
which is also a finite rank operator (because $\ra(A)^\perp=\kr(A^\dagger)$ is finite dimensional as $A^\dagger$ is Fredholm). Therefore, $SA-\Iop$ and $AS-\mathbb{I}_{\ur}$ are finite rank operators and hence compact.\\
($\Leftarrow$) Now suppose that $SA-\Iop=F_1$ and $AS-\mathbb{I}_{\ur}=F_2$ with $F_1$ and $F_2$ being finite rank operators on $\vr$ and $\ur$ respectively. By proposition \ref{PD22}, $F_1$ and $F_2$ are compact operators. Therefore by theorem \ref{FT3}, $A$ is a Fredholm operator.
\end{proof}
\begin{definition}\label{FD3}
Let $V_0,\cdots,V_n$ be right quaternionic vector spaces and let $A_j:V_j\longrightarrow V_{j+1}, 0\leq j\leq n-1,$ be right linear operators. Then the sequence
$$V_0\xrightarrow{A_0} V_1\xrightarrow{A_1}V_2\xrightarrow{A_2}\cdots\xrightarrow{A_{n-2}}V_{n-1}\xrightarrow{A_{n-1}}V_n$$
is called exact if $\ra(A_j)=\kr(A_{j+1})$ for $j=0,1,\cdots n-2.$
\end{definition}
\begin{lemma}\label{FL1}
Let $V_j$ be a finite dimensional right quaternionic vector space for each $j=0,1,\cdots, n$. If
$$\{0\}=V_0\xrightarrow{A_0} V_1\xrightarrow{A_1}V_2\xrightarrow{A_2}\cdots\xrightarrow{A_{n-2}}V_{n-1}\xrightarrow{A_{n-1}}V_n=\{0\}$$
be an exact sequence with $\dim(V_j)<\infty$ for all $j=0,1,\cdots,n$, then
$$\sum_{j=0}^{n-1}(-1)^j\dim(V_j)=0.$$
\end{lemma}
\begin{proof}
For each $j$, decompose $V_j$ as $V_j=\kr(A_j)\oplus Y_j$, where $Y_j$ is an orthogonal complement of $\kr(A_j)$. The exactness of the sequence implies that $A_j:Y_j\longrightarrow\kr(A_{j+1})=\ra(A_j)$ is an isomorphism. Therefore, $\dim(Y_j)=\dim(\kr(A_{j+1}))$ for each $j=0,1,\cdots,n-1$. Hence $\dim(V_j)=\dim(\kr(A_j))+\dim(\kr(A_{j+1})).$ Further, $\dim(\kr(A_0))=0$ and $\dim(V_{n-1})=\dim(\kr(A_{n-1}))$. Hence the result follows.
\end{proof}
\begin{theorem}\label{FT2}
Let $\vr, \ur$ and $W_\quat^R$ be right quaternionic Hilbert spaces. If $A_1\in\B(\vr,\ur)$ and $A_2\in\B(\ur,W_\quat^R)$ are two Fredholm operators, then $A_2A_1\in\B(\vr, W_\quat^R)$ is also a Fredholm operator, and it satisfies $\ind(A_2A_1)=\ind(A_1)+\ind(A_2)$.
\end{theorem}
\begin{proof}
Suppose that $A_1$ and $A_2$ be Fredholm operators. Clearly, $\kr(A_1)\subseteq\kr(A_2A_1)$ and $A_1:\kr(A_2A_1)\longrightarrow\kr(A_2)$. Therefore, there is an isomorphism $$S:\kr(A_2A_1)/\kr(A_1)\longrightarrow B\subseteq\kr(A_2),$$ where $B$ is a subspace of $\kr(A_2)$. Hence, $\dim(\kr(A_2A_1)/\kr(A_1))\leq\dim(\kr(A_2))<\infty$, and thus
$$\dim(\kr(A_2A_1))\leq\dim(\kr(A_1))+\dim(\kr(A_2))<\infty.$$
Since $\ra(A_2A_1)\subseteq\ra(A_2)\subseteq W_\quat^R$, we have
$$\left(W_\quat^R/\ra(A_2A_1)\right)/\left(\ra(A_2)/\ra(A_2A_1)\right)=W_\quat^R/\ra(A_2).$$
Hence,
\begin{equation}\label{FE3}
\dim(\cokr(A_2A_1))=\dim(\ra(A_2)/\ra(A_2A_1))+\dim(\cokr(A_2)).
\end{equation}
Further, $A_2:\ur/\ra(A_1)\longrightarrow \ra(A_2)/\ra(A_2A_1)$ is surjective, and therefore
\begin{equation}\label{FE4} \dim(\ra(A_2)/\ra(A_2A_1))\leq\dim(\cokr(A_1))<\infty.
\end{equation}
From equations \ref{FE3},\ref{FE4}, we have 
$$\dim(\cokr(A_2A_1))\leq\dim(\cokr(A_1))+\dim(\cokr(A_2))<\infty.$$
Thus $A_2A_1$ is a Fredholm operator.\\
Let $\mathfrak{i}:\kr(A_1)\longrightarrow\kr(A_2A_1)$ be the inclusion map, $Q:\kr(A_2)\longrightarrow\ur/\ra(A_1)$ be the quotient map, and $E:W_\quat^R/\ra(A_2A_1)\longrightarrow W_\quat^R/\ra(A_2)$ maps equivalence classes modulo $\ra(A_2A_1)$ to equivalence classes modulo $\ra(A_2)$. Consider the sequence
\begin{eqnarray*}
\{0\}&\xrightarrow{O_1}\kr(A_1)&\xrightarrow{\mathfrak{i}}\kr(A_2A_1)\xrightarrow{A_1}\kr(A_2)\xrightarrow{Q}\ur/\ra(A_1)\rightarrow\\
& &\quad\quad\quad\xrightarrow{A_2}W_\quat^R/\ra(A_2A_1)\xrightarrow{E} W_\quat^R/\ra(A_2)\xrightarrow{O_2}\{0\}.
\end{eqnarray*}
Since $\ra(O_1)=\kr(\mathfrak{i})=\{0\}$, $\ra(\mathfrak{i})=\kr(A_1)$, $\ra(A_1)=\kr(Q)$, $\ra(Q)=\kr(A_2)$ and $\ra(A_2)=\kr(E)$, the above sequence is exact. Therefore, by lemma \ref{FL1}, we have
\begin{eqnarray*}
0&=&\dim(\{0\})-\dim(\kr(A_1))+\dim(\kr(A_2A_1))-\dim(\kr(A_2))\\
& &\quad+\dim\left(\ur/\ra(A_1)\right)-\dim\left(W_\quat^R/\ra(A_2A_1)\right)+\dim\left(W_\quat^R/\ra(A_2)\right)\\
&=&[\dim(\kr(A_2A_1))-\dim\left(W_\quat^R/\ra(A_2A_1)\right)]-[\dim(\kr(A_1))-\dim\left(\ur/\ra(A_1)\right)]\\
& &\quad\quad-[\dim(\kr(A_2))-\dim\left(W_\quat^R/\ra(A_2)\right)]\\
&=&\ind(A_2A_1)-\ind(A_1)-\ind(A_2).
\end{eqnarray*}
Hence the result follows.
\end{proof}
\begin{corollary}\label{FC1}
Let $A\in\B(\vr,\ur)$ and $B\in\B(\ur,\vr)$ such that $AB=F$ and $BA=G$ are Fredholm operators. Then $A$ and $B$ are Fredholm operators and $\ind(AB)=\ind(A)+\ind(B)$.
\end{corollary}
\begin{proof}
Since $\ra(F)\subseteq\ra(A)$ and $\kr(A)\subseteq\kr(G)$, we have $\dim(\kr(A))<\infty$ and $\dim(\ur/\ra(A))\leq\dim(\ur/\ra(F))<\infty$. Hence $A$ is Fredholm. Similarly $B$ is also Fredholm. The index equality follows from theorem \ref{FT2}.
\end{proof}
\begin{lemma}\label{FL2}
Let $F\in\B(\vr)$ be a finite rank operator, then $\ind(\Iop+F)=0.$
\end{lemma}
\begin{proof}
Suppose that $F\in\B(\vr)$ be a finite rank operator. Then, $\dim(\ra(F))<\infty$. Also by theorem \ref{PD22}, $\kr(F)^\perp=\ra(F^\dagger)$ is also a finite dimensional subspace. Define $L:=\ra(F)+\kr(F)^\perp$ with $\dim(L)<\infty$, and hence $L$ is closed. Thus $\vr=L\oplus L^\perp$. Further,
\begin{eqnarray*}
(\Iop+F)L&\subseteq&L+FL\subseteq L,\\
(\Iop+F)\rvert_{L^\perp}&=&\mathbb{I}_{L^\perp}.
\end{eqnarray*}
Hence $L$ and $L^\perp$ are invariant under $\Iop+F$, and therefore,
$$\ind(\Iop+F)=\ind((\Iop+F)\rvert_L)+\ind((\Iop+F)\rvert_{L^\perp}).$$
Since $\kr((\Iop+F)\rvert_{L^\perp})=\{0\}$ and $L^\perp/\ra((\Iop+F)\rvert_{L^\perp})\cong\{0\}$, $\ind((\Iop+F)\rvert_{L^\perp})=0.$ Further, since $\dim(L)<\infty$, $\dim(L)=\dim(\kr((\Iop+F)\rvert_{L}))+\dim(\ra((\Iop+F)\rvert_{L}))$, and therefore $\text{codim}((\Iop+F)\rvert_{L}))=\dim(\kr((\Iop+F)\rvert_{L}))$. Hence, $\ind(\Iop+F)\rvert_{L})=0$, which concludes the proof.
\end{proof}
\begin{theorem}\label{FT4}
Let $A\in\B(\vr,\ur)$ be a Fredholm operator, then for any compact operator $K\in\B(\vr,\ur)$, $A+K$ is a Fredholm operator and $\ind(A+K)=\ind(A)$.
\end{theorem}
\begin{proof}
Let $A\in\B(\vr,\ur)$ be a Fredholm operator, then by theorem \ref{FT3}, there exists $S_1,S_2\in\B(\ur,\vr)$ and compact operators $K_1, K_2$ on $\vr$, $\ur$ respectively such that $S_1A=\Iop+K_1$ and $AS_2=\mathbb{I}_{\ur}+K_2$. Hence we have
\begin{eqnarray*}
S_1(A+K)&=&S_1A+S_1K=\Iop+K_1+S_1K=\Iop+K_1'\\
(A+K)S_2&=&AS_2+KS_2=\mathbb{I}_{\ur}+K_2+KS_2=\mathbb{I}_{\ur}+K_2',
\end{eqnarray*}
where, by proposition \ref{com1}, $K_1'=K_1+S_1K$ and $K_2'=K_2+KS_2$ are compact operators. Therefore by theorem \ref{FT3}, $A+K$ is a Fredholm operator.\\
Now by proposition \ref{FP3}, there exists $S\in\B(\ur,\vr)$ and finite rank operators $F_1$ and $F_2$ respectively on $\vr$ and $\ur$ such that
$$SA=\Iop+F_1\quad\text{and}\quad AS=\mathbb{I}_{\ur}+F_2.$$
Hence by proposition \ref{PD22} and theorem \ref{FT1}, $SA$ and $AS$ are Fredholm operators. Thus by corollary \ref{FC1}, $S$ is a Fredholm operator. Hence by lemma \ref{FL2} and theorem \ref{FT2},
\begin{equation}\label{FE5}
0=\ind(SA)=\ind(S)+\ind(A).
\end{equation}
Since $S$ and $A+K$ are Fredholm operators, by theorem \ref{FT2}, $S(A+K)$ is a Fredholm operator. Therefore by theorem \ref{FP3}, there exists a finite rank operator $F_3$ on $\vr$ such that $S(A+K)=\Iop+F_3.$ Thus by lemma \ref{FL2} and theorem \ref{FT2}, we have
\begin{equation}\label{FE6} 
0=\ind(S(A+K))=\ind(S)+\ind(A+K).
\end{equation}
From equations \ref{FE5} and \ref{FE6}, we have $\ind(A+K)=\ind(A)$.
\end{proof}
\begin{corollary}
Let $K\in\B(\vr)$ be compact then $\Iop+K$ is Fredholm and $\ind(I+K)=0.$
\end{corollary}
\begin{proof} Let $K\in\B_0(\vr).$
Clearly $\Iop$ is Fredholm and $\ind(\Iop)=0$. Therefore, by theorem \ref{FT4}, $\Iop+K$ is Fredholm and $\ind(\Iop+K)=\ind(\Iop)=0$.
\end{proof}
\begin{corollary}\label{CI1} Every invertible operator $A\in\B(\vr)$ is Fredholm and $\ind(A)=0$.
\end{corollary}
\begin{proof}
Let $A\in\B(\vr)$ is invertible. Then $AA^{-1}=A^{-1}A=\Iop$ is Fredholm. Therefore, by corollary \ref{FC1}, $A$ and $A^{-1}$ are Fredholm operators. Since $A$ is invertible so is $A^\dagger$. Therefore, $\kr(A)=\kr(A^\dagger)=\{0\}$. Hence $\ind(A)=0$.
\end{proof}
\begin{corollary}\label{CN}
Let $n$ be a non-negative integer. If $A\in\FF(\vr)$ then $A^n\in\FF(\vr)$ and $\ind(A^n)=n~\ind(A)$.
\end{corollary}
\begin{proof}
The result is trivial for $n=0,1$. $n=2$ follows from theorem \ref{FT2}, and by induction we have the result.
\end{proof}
\begin{theorem}\label{FT5}
Let $A\in\B(\vr,\ur)$ be a Fredholm operator. Then there exists a constant $c>0$ such that for every operator $S\in\B(\vr,\ur)$ with $\|S\|<c$, $A+S$ is a Fredholm operator and $\ind(A+S)=\ind(A)$.
\end{theorem}
\begin{proof}
Let  $A\in\B(\vr,\ur)$ be a Fredholm operator, then by proposition \ref{FP2}, $A^\dagger\in\B(\ur,\vr)$ is Fredholm. Therefore, $\kr(A)$ and $\kr(A^\dagger)$ are finite dimensional subspaces of $\vr$ and $\ur$ respectively. Let $R_1, R_2\in\B(\ur,\vr)$ be operators such that
$R_1A=\Iop-P_k$ and $AR_2=\mathbb{I}_{\ur}-P_{k^\dagger}$, where $P_k:\vr\longrightarrow \kr(A)$ and $P_{k^\dagger}:\ur\longrightarrow\kr(A^\dagger)$ be orthogonal projections. Further $P_k$ and $P_{k^\dagger}$ are finite rank operators hence compact. Now $R_1(A+S)=R_1A+R_1S=\Iop+R_1S-P_k$. For $\|S\|<\min\{\|R_1\|^{-1},\|R_2\|^{-1}\}$, we have $\|R_1S\|<1$. Therefore by proposition \ref{PP3}, $\Iop+R_1S$ is invertible and by the open mapping theorem $(\Iop+R_1S)^{-1}\in\B(\vr)$. Now
$$(\Iop+R_1S)^{-1}R_1(A+S)=\Iop-(\Iop+R_1S)^{-1}P_k=\Iop+K_1,$$
where, by proposition \ref{com1}, $K_1=-(\Iop+R_1S)^{-1}P_k$ is a compact operator on $\vr$. Further, $(A+S)R_2=AR_2+SR_2=\mathbb{I}_{\ur}+SR_2-P_{k^\dagger}$. Again $\|SR_2\|<1$, therefore by the same argument as above $\mathbb{I}_{\ur}+SR_2$ is invertible and $(\mathbb{I}_{\ur}+SR_2)^{-1}\in\B(\ur)$. Hence.
$$(A+S)R_2(\mathbb{I}_{\ur}+SR_2)^{-1}=\mathbb{I}_{\ur}-(\mathbb{I}_{\ur}+SR_2)^{-1}P_{k^\dagger}=\mathbb{I}_{\ur}+K_2,$$ where $K_2=-(\mathbb{I}_{\ur}+SR_2)^{-1}P_{k^\dagger}$ is a compact operator on $\ur$. That is, we have operators $S_1=(\Iop+R_1S)^{-1}R_1, S_2=R_2(\mathbb{I}_{\ur}+SR_2)^{-1}\in\B(\ur,\vr)$ and compact operators $K_1$ and $K_2$, on $\vr$ and $\ur$ respectively, such that
$$S_1(A+S)=\Iop+K_1\quad\text{and}\quad (A+S)S_2=\mathbb{I}_{\ur}+K_2.$$
Therefore, by theorem \ref{FT3}, $A+S$ is a Fredholm operator. The identity, $\ind(A+S)=\ind(A)$ follows similar to the proof in theorem \ref{FT4}.
\end{proof}
\begin{theorem}\label{FT6}
An operator $A\in\B(\vr)$ is left semi-Fredholm if and only if $\ra(A)$ is closed and $\kr(A)$ is finite dimensional. Hence
\begin{eqnarray}
\FF_l(\vr)&=&\{A\in\B(\vr)~~|~~\ra(A)~\text{ is closed and}~ \dim(\kr(A))<\infty\}\\
\FF_r(\vr)&=&\{A\in\B(\vr)~~|~~\ra(A)~\text{ is closed and}~ \dim(\kr(A^\dagger))<\infty\}
\end{eqnarray}
\end{theorem}
\begin{proof}
($\Rightarrow$) Suppose that $A\in\B(\vr)$ is left semi-Fredholm. Then there exists $S\in\B(\vr)$ and a compact operator $K$ in $\vr$ such that $SA=\Iop+K$. We have
$\kr(A)\subseteq\kr(SA)=\kr(\Iop+K)$ and
$\ra(SA)=\ra(\Iop+K)$. By theorem \ref{FT1}, $\Iop+K$ is Fredholm and hence $\dim(\kr(\Iop+K))<\infty$ and which implies $\dim(\kr(A))<\infty$. By proposition \ref{FP1}, $\ra(\Iop+K)$ is closed. That is, $\dim(\kr(A))<\infty$, $\dim(A(\kr(SA)))<\infty$ and $\ra(SA)$ is closed.Consider the operator,
$$(SA)\vert_{\kr(SA)^\perp}:\kr(SA)^\perp\longrightarrow\vr.$$
Since this operator is injective with a closed range, by theorem \ref{NT1}, it is bounded below. Hence, there exists $c>0$ such that
$$c\|\phi\|_{\vr}\leq \|SA\phi\|_{\vr}\leq\|S\|\|A\phi\|_{\vr}\quad\text{for all}~~\phi\in\kr(SA)^\perp.$$ Hence,
$$A\vert_{\kr(SA)^\perp}:\kr(SA)^\perp\longrightarrow\vr$$ is bounded below, therefore by theorem \ref{NT1}, $A\vert_{\kr(SA)^\perp}$ has a closed range. Since $\vr=\kr(SA)\oplus\kr(SA)^\perp$, we have
$$\ra(A)=A(\vr)=A(\kr(SA)\oplus\kr(SA)^\perp)=A(\kr(SA))\oplus A(\kr(SA)^\perp)),$$
which is closed because the direct sum of a closed subspace and a finite dimensional space is closed.\\
($\Leftarrow$) Suppose that $\dim(\kr(A))<\infty$ and $\ra(A)$ is closed. Hence, $A\vert_{\kr(A)^\perp}:\kr(A)^\perp\longrightarrow\vr$ is injective and has closed range, $\ra(A\vert_{\kr(A)^\perp})=\ra(A)$. Therefore by theorem \ref{NT1}, it has a bounded inverse. Let $P_r:\vr\longrightarrow\ra(A)$ be the orthogonal projection onto $\ra(A)$. Define $T\in\B(\vr)$ by $T=(A\vert_{\kr(A)^\perp})^{-1}P_r$. If $\phi\in\kr(A)$, then $TA\phi=0$. If $\psi\in\kr(A)^\perp$, then
$$TA\psi=(A\vert_{\kr(A)^\perp})^{-1}P_r(A\vert_{\kr(A)^\perp})\psi=(A\vert_{\kr(A)^\perp})^{-1}(A\vert_{\kr(A)^\perp})\psi=\psi.$$
Therefore, for every $\eta=\phi+\psi\in\vr=\kr(A)\oplus\kr(A)^\perp$,
$$TA\eta=TA\phi+TA\psi=TA\psi=\psi=P_{n}\psi,$$
where $P_n:\vr\longrightarrow\kr(A)^\perp$ is the orthogonal projection onto $\kr(A)^\perp$. Thus, $$TA=P_n=\Iop-(\Iop-P_n)=\Iop+K,$$
where $K=-(\Iop-P_n)$ is a finite rank operator as it is an orthogonal projection onto the finite dimensional subspace $\kr(A)$. That is, $TA=\Iop+K$ and $K$ is compact, and hence $A$ is left semi-Fredholm. Therefore,
\begin{equation}\label{FE7}
\FF_l(\vr)=\{A\in\B(\vr)~~|~~\ra(A)~\text{ is closed and}~ \dim(\kr(A))<\infty\}.
\end{equation}
Now by remark \ref{FR4}, $A^\dagger\in\FF_l(\vr)\Leftrightarrow A\in\FF_r(\vr)$, and by proposition \ref{NP1}, $\ra(A)$ is closed if and only if $\ra(A^\dagger)$ is closed. Hence, from equation \ref{FE7}, we have
$$\FF_r(\vr)=\{A\in\B(\vr)~~|~~\ra(A)~\text{ is closed and}~ \dim(\kr(A^\dagger))<\infty\}.$$
\end{proof}
\begin{remark}\label{ER2}
Let $A\in\B(\vr)$.
\begin{enumerate}
\item [(a)] The so-called Weyl operators are Fredholm operators on $\vr$ with null index. That is, the set of all Weyl operators,
$$\Wy(\vr)=\{A\in\FF(\vr)~~|~~\ind(A)=0\}.$$
\item[(b)] Since, by remark \ref{FR4} (e), $A\in\FF(\vr)\Leftrightarrow A^\dagger\in\FF(\vr)$ and $\ind(A)=-\ind(A^\dagger)$, $A\in\Wy(\vr)\Leftrightarrow A^\dagger\in\Wy(\vr)$.
\item[(c)] By proposition \ref{FT1} and lemma \ref{FL2}, if $F$ is a finite rank operator, then $\Iop+F\in\Wy(\vr)$. 
\item[(d)]By theorem \ref{FT2}, $A,B\in\Wy(\vr)\Rightarrow AB\in\Wy(\vr)$.
\item[(e)]By theorem \ref{FT4}, $A\in\Wy(\vr), K\in\B_0(\vr)\Rightarrow A+K\in\Wy(\vr).$
\item[(f)] By corollary \ref{CI1}, $A\in\B(\vr)$ is invertible, then $A\in\Wy(\vr)$
\item[(g)]Suppose $\dim(\vr)<\infty$, then $\dim(\vr)=\dim(\kr(A))+\dim(\ra(A))$. Also $\vr=\overline{\ra(A)}+\ra(A)^\perp$. Hence $\dim(\vr)=\dim(\ra(A))+\dim(\ra(A)^\perp)$. Thus,
\begin{eqnarray*}
\ind(A)&=&\dim(\kr(A))-\dim(\kr(A^\dagger))\\
&=&\dim(\kr(A))-\dim(\ra(A)^\perp)\\
&=&\dim(\kr(A))+\dim(\ra(A))-\dim(\vr)=0.
\end{eqnarray*}
Therefore, in the finite dimensional space, every operator in $\B(\vr)$ is a Fredholm operator with index zero. In this case, $\Wy(\vr)=\B(\vr)$.
\end{enumerate}
\end{remark}

\begin{remark}\label{FR3}
In the complex case, the Fredholm theory can also be extended to unbounded operators using the graph norm. For such a complex treatment we refer the reader to, for example, \cite{Goh}. Following the complex theory we may be able to extend the quaternionic Fredholm theory to quaternionic unbounded operators.
\end{remark}
\section{essential S-spectrum}
In this section we define and study the essential S-spectrum in $\vr$ in terms of the Fredholm operators. In particular, we give an interesting characterization to the S-spectrum in terms of the Fredholm operators and the Fredholm index (see proposition \ref{EP7}).
\begin{theorem}\cite{ghimorper}\label{ET0}
Let $\vr$ be a right quaternionic Hilbert space equipped with a left scalar multiplication. Then the set $\B(\vr)$ equipped with the point-wise sum, with the left and right scalar multiplications defined in equations \ref{lft_mul-op} and \ref{rgt_mul-op}, with the composition as product, with the adjunction $A\longrightarrow A^\dagger$, as in \ref{Ad1}, as $^*-$ involution and with the norm defined in \ref{PE1}, is a quaternionic two-sided Banach $C^*$-algebra with unity $\Iop$.
\end{theorem}
\begin{remark}\label{ER0}
In the above theorem, if the left scalar multiplication is left out on $\vr$, then $\B(\vr)$ becomes a real Banach $C^*$-algebra with unity $\Iop$.
\end{remark}
\begin{theorem}\label{ET1}
The set of all compact operators, $\B_0(\vr)$ is a closed biideal of $\B(\vr)$ and is closed under adjunction.
\end{theorem}
\begin{proof}
The theorem \ref{ET0} gives together with proposition \ref{com1} that $\B_0(\vr)$ is a closed biideal of $\B(\vr)$. Theorem \ref{comeq} is enough to conclude that $\B_0(\vr)$ is closed under adjunction. Hence the theorem holds.
\end{proof}
On the quotient space $B(\vr)/B_0(\vr)$ the coset of $A\in\B(\vr)$ is
$$[A]=\{S\in\B(\vr)~|~S=A+K~~~~\text{for some}~~~K\in\B_0(\vr)\}=A+\B_0(\vr).$$
On the quotient space define the product
$$[A][B]=[AB].$$
Since $\B_0(\vr)$ is a closed subspace of $\B(\vr)$, with the above product, $\B(\vr)/\B_0(\vr)$ is a unital Banach algebra with unit $[\Iop]$. We call this algebra the quaternionic Calkin algebra. Define the natural quotient map
$$\pi:\B(\vr)\longrightarrow \B(\vr)/\B_0(\vr)\quad\text{by}\quad \pi(A)=[A]=A+\B_0(\vr).$$
Note that $[0]=\B_0(\vr)$ and hence
$$\kr(\pi)=\{A\in\B(\vr)~~|~~\pi(A)=[0]\}=\B_0(\vr).$$
Since $\B_0(\vr)$ is an ideal of $\B(\vr)$, for $A,B\in\B(\vr)$, we have
\begin{enumerate}
\item[(a)]$\pi(A+B)=(A+B)+\B_0(\vr)=(A+\B_0(\vr))+(B+\B_0(\vr))=\pi(A)+\pi(B)$.
\item[(b)]$\pi(AB)=AB+\B_0(\vr)=(A+\B_0(\vr))(B+\B_0(\vr))=\pi(A)\pi(B).$
\item[(c)]$\pi(\Iop)=[\Iop]$.
\end{enumerate}
Hence $\pi$ is a unital homomorphism. The norm on $\B(\vr)/\B_0(\vr)$ is given by
$$\|[A]\|=\inf_{K\in\B_0(\vr)}\|A+K\|\leq\|A\|.$$
Therefore $\pi$ is a contraction.
\begin{theorem}\label{ET2}
Let $A\in\B(\vr)$. The following are pairwise equivalent.
\begin{enumerate}
\item [(a)]$A\in\mathcal{F}_l(\vr)$.
\item[(b)]There exist $S\in\B(\vr)$ and $K\in\B_0(\vr)$ such that $SA=\Iop+K$.
\item[(c)]There exist $S\in\B(\vr)$ and $K_1,K_2,K_3\in\B_0(\vr)$ such that $(S+K_1)(A+K_2)=\Iop+K_3$.
\item[(d)]$\pi(S)\pi(A)=\pi(\Iop)$, the identity in  $\B(\vr)/\B_0(\vr)$, for some $\pi(S)\in \B(\vr)/\B_0(\vr)$.
\item[(e)] $\pi(A)$ is left invertible in  $\B(\vr)/\B_0(\vr)$.
\end{enumerate}
\end{theorem}
\begin{proof}
By the definition of left semi-Fredholm operators (a) and (b) are equivalent.\\
(b)$\Rightarrow (c)$: Suppose (b) holds. Since $(S+K_1)(A+K_2)=SA+SK_2+K_1A+K_1K_2=\Iop+K+SK_2+K_1A+K_1K_2=\Iop+K_3$,
where, by proposition \ref{com1}, $K_3=K+SK_2+K_1A+K_1K_2$ is compact.\\
(c)$\Rightarrow$ (d): Suppose (c) holds. Then there exists $S_1\in[S]=\pi(S), A_1\in[A]=\pi(A)$ and $J\in[\Iop]=\pi(\Iop)$ such that $S_1A_1=J$. Hence $\pi(S)\pi(A)=[S][A]=[SA]=[J]=[\Iop]=\pi(\Iop).$
Therefore (d) is established.\\
(d)$\Rightarrow$ (b): Suppose (d) holds. $B\in\pi(S)\pi(A)=\pi(SA)$ if and only if $B=SA+K_1$ for some $K_1\in\B_o(\vr)$. Also $B\in\pi(\Iop)=[\Iop]$ if and only if $B=\Iop+K_2$ for some $K_2\in\B_o(\vr)$. Therefore, $SA=\Iop+K_2-K_1=\Iop+K$, where $K=K_2-K_1\in\B_0(\vr)$. Thus $A\in\mathcal{F}_l(\vr)$.\\
(a) and (e) are equivalent by the definition of left invertibility.
\end{proof}
\begin{corollary}\label{EC1}
Let $A\in\B(\vr)$. $A\in\mathcal{F}_l(\vr)$ (or $A\in\mathcal{F}_r(\vr)$) if and only if $\pi(A)$ is left (or right) invertible in the Calkin algebra  $\B(\vr)/\B_0(\vr)$.
\end{corollary}
\begin{proof}
It is straightforward from the theorem \ref{ET2} for $\mathcal{F}_l(\vr)$. Also similar version of theorem \ref{ET2} holds for $\mathcal{F}_r(\vr)$.
\end{proof}
\begin{definition}\label{ED1}
The essential $S$-spectrum (or the Calkin $S$-spectrum) $\se(A)$ of $A\in\B(\vr)$ is the $S$-spectrum of $\pi(A)$ in the unital Banach algebra $B(\vr)/\B_0(\vr)$. That is, $$\se(A)=\sigma_S(\pi(A)).$$ Similarly, the left essential $S$-spectrum $\sel(A)$ and the right essential $S$-spectrum $\ser(A)$ are the left and right $S$-spectrum of $\pi(A)$ respectively. That is,
$$\sel(A)=\sigma_l^S(\pi(A))\quad\text{and}\quad \ser(A)=\sigma_r^S(\pi(A))$$ in  $B(\vr)/\B_0(\vr)$.\\
Clearly, by definition, $\se(A)=\sel(A)\cup\ser(A)$ and $\se(A)$ is a compact subset of $\quat$.
\end{definition}
\begin{proposition}\label{EP1}
Let $A\in\B(\vr)$, then
\begin{eqnarray}
\sel(A)&=&\{\qu\in\quat~|~R_\qu(A)\in\B(\vr)\setminus\mathcal{F}_l(\vr)\}\label{EE1}\\
\ser(A)&=&\{\qu\in\quat~|~R_\qu(A)\in\B(\vr)\setminus\mathcal{F}_r(\vr)\}\label{EE2}
\end{eqnarray}
\end{proposition}
\begin{proof}
Let $A\in\B(\vr)$. Then by corollary \ref{EC1}, $R_\qu(A)\not\in\mathcal{F}_l(\vr)$ if and only if $\pi(R_\qu(A))$ is not left invertible in $B(\vr)/\B_0(\vr)$, which means, by the definition of left $S$-spectrum, $\qu\in\sigma_l^S(\pi(A))$. That is,
$\qu\in\sel(A)$ if and only if $R_\qu(A)\not\in\mathcal{F}_l(\vr)$. Hence we have equation \ref{EE1}. A similar argument proves equation \ref{EE2}.
\end{proof}
\begin{corollary}(Atkinson theorem)\label{EC2}
Let $A\in\B(\vr)$, then
\begin{equation}\label{EE3}
\se(A)=\{\qu\in\quat~~~|~~~R_\qu(A)\in\B(\vr)\setminus\mathcal{F}(\vr)\}.
\end{equation}
\end{corollary}
\begin{proof}
Let $A\in\B(\vr)$. Since $\se(A)=\sel(A)\cup\ser(A)$ and $\FF(\vr)=\FF_l(\vr)\cap\FF_r(\vr)$, it is straightforward from proposition \ref{EP1}.
\end{proof}
\begin{corollary}(Atkinson theorem)\label{EC3}
Let $A\in\B(\vr)$. $A$ is Fredholm if and only if $\pi(A)$ is invertible in the Calkin algebra $B(\vr)/\B_0(\vr)$.
\end{corollary}
\begin{proof}
$A\in\B(\vr)$ is Fredholm if and only if $A$ is left and right semi-Fredholm. Further, $\pi(A)$ is invertible in $B(\vr)/\B_0(\vr)$ if and only if $\pi(A)$ is both left and right invertible in $B(\vr)/\B_0(\vr)$. Therefore, from corollary \ref{EC1}, we have the desired result.
\end{proof}
\begin{proposition}\label{EP2}
Let $A\in\B(\vr)$, then $\sel(A)$ and $\ser(A)$ are closed subsets of $\quat$.
\end{proposition}
\begin{proof}
In the unital algebra $\B(\vr)$, by theorems \ref{T1} and \ref{T2}, $\sigma_l^S(A)=\apo(A)$ and $\sigma_r^S(A)=\overline{\apo(A^\dagger)}$ are closed subsets of $\quat$. By the same argument, in the Calkin algebra $B(\vr)/\B_0(\vr)$, $\sel(A)=\sigma_l^S(\pi(A))$ and $\ser(A)=\sigma_r^S(\pi(A))$ are closed subsets of $\quat$.
\end{proof}
\begin{proposition}\label{EP3}
Let $A\in\B(\vr)$, then
\begin{eqnarray}
\sel(A)&=&\{\qu\in\quat~~|~~\ra(R_\qu(A))~\text{is not closed or}~~\dim(\kr(R_\qu(A)))=\infty\}.\label{EE4}\\
\ser(A)&=&\{\qu\in\quat~~|~~\ra(R_\qu(A))~\text{is not closed or}~~\dim(\kr(R_{\oqu}(A^\dagger)))=\infty\}\label{EE5}.
\end{eqnarray}
\end{proposition}
\begin{proof}
Let $A\in\B(\vr)$. By proposition \ref{EP1}, $\qu\in\sel(A)$ if and only if $R_\qu(A)$ is not left semi-Fredholm. By theorem \ref{FT6}, $R_\qu(A)$ is not left semi-Fredholm if and only if $\ra(R_\qu(A))$ is not closed or $\dim(\kr(R_\qu(A)))=\infty$. Hence we have equation \ref{EE4}. In the same way equation \ref{EE5} can be obtained.
\end{proof}
\begin{corollary}\label{EC4}
Let $A\in\B(\vr)$, then
$$\sel(A)=\ser(A^\dagger),\quad \ser(A)=\sel(A^\dagger)\quad\text{and hence}\quad\se(A)=\se(A^\dagger).$$
\end{corollary}
\begin{proof}Since $R_\qu(A)=R_{\oqu}(A)$, by proposition \ref{NP1}, $\ra(R_\qu(A))$ is closed if and only if $\ra(R_{\oqu}(A^\dagger))$ is closed, and $\se(A)=\ser(A)\cup\sel(A)$, proposition \ref{EP3} conclude the results.
\end{proof}
\begin{remark}\label{ER1}
If $\dim(\vr)<\infty$, then all operators in $\B(\vr)$ are finite rank operators and hence compact. Therefore, the Calkin algebra $B(\vr)/\B_0(\vr)$ is null. Hence, in this case, $\se(A)=\emptyset$.
\end{remark}
\begin{proposition}\label{EP4}
For $A\in\B(\vr)$, $\se(A)\not=\emptyset$ if and only if $\dim(\vr)=\infty.$
\end{proposition}
\begin{proof}
($\Leftarrow$) Suppose $\dim(\vr)=\infty$. $\B(\vr)/\B_0(\vr)$ is a unital Banach algebra. By proposition \ref{PP1}, the $S$-spectrum on a unital Banach algebra is not empty. Therefore $\se(A)=\sigma_S(\pi(A))\not=\emptyset.$\\
($\Rightarrow$) By remark \ref{ER1}, $\dim(\vr)<\infty$ implies $\se(A)=\emptyset.$ That is, $\se(A)\not=\emptyset$ implies $ \dim(\vr)=\infty.$
\end{proof}
\begin{proposition}\label{EP5}
For every $A\in\B(\vr)$ and $K\in\B_0(\vr)$, we have $\se(A+K)=\se(A)$. In the same way, $\sel(A+K)=\sel(A)$ and $\ser(A+K)=\ser(A)$.
\end{proposition}
\begin{proof}
For every $A\in\B(\vr)$ and $K\in\B_0(\vr)$ we have $\pi(A+K)=\pi(A)$ in $\B(\vr)/B_0(\vr)$. Therefore by the definition, we get
$$\se(A+K)=\sigma_S(\pi(A+K))=\sigma_S(\pi(A))=\se(A).$$
The other two equalities follow in the same way.
\end{proof}
\begin{proposition}\label{EP6}
Let $A\in\B(\vr)$. Then,
$$\sel(A)\subseteq\sigma_l^S(A)\quad\text{and}\quad\ser(A)\subseteq\sigma_r^S(A),$$
and hence $\se(A)\subseteq\sigma_S(A)$.
\end{proposition}
\begin{proof}
It is straightforward from proposition \ref{SP1} and \ref{EP3}.
\end{proof}
\begin{definition}\label{ED2}
Let $A\in\B(\vr)$ and $k\in\Z\setminus\{0\}$. Define,
$$\sigma_k^S(A)=\{\qu\in\quat~~|~~R_\qu(A)\in\FF(\vr)\quad\text{and}\quad \ind(R_\qu(A))=k\}.$$
Also
$$\sigma_0^S=\{\qu\in\sigma_S(A)~~|~~R_\qu(A)\in\Wy(\vr)\}.$$
\end{definition}
\begin{proposition}\label{EP7}
Let $A\in\B(\vr)$, then $\displaystyle\sigma_S(A)=\se(A)\cup\bigcup_{k\in\Z}\sigma_k^S(A).$
\end{proposition}
\begin{proof}
Clearly the family $\{\sigma_k^S(A)\}_{k\in\Z}$ is pair-wise disjoint. Let $0\not=k\in\Z$ and $\qu\in\sigma_k^S(A)$. If $k>0$, then $0<\ind(R_\qu(A))<\infty$, hence $0<\dim(\kr(R_\qu(A)))-\dim(\kr(R_\qu(A^\dagger)))<\infty$. Thus, $\ker(R_\qu(A))\not=\{0\}$ and therefore $q\in\sigma_S(A)$. If $k<0$, then $\dim(\kr(R_\qu(A)))<\dim(\kr(R_\qu(A^\dagger)))$. Hence $\kr(R_\qu(A^\dagger))\not=\{0\}$ and therefore $R_\qu(A^\dagger)$ is not invertible. Thus, by proposition \ref{NP2}, $R_\qu(A)$ is not invertible. That is, $\qu\in\sigma_S(A)$. Altogether we get
$$\bigcup_{k\in\Z}\sigma_k^S(A)=\{\qu\in\sigma_S(A)~~|~~R_{\qu}(A)\in\FF(\vr)\}.$$
Also, by proposition \ref{EP6}, we have, $\se(A)\subseteq\sigma_S(A)$ and by corollary \ref{EC2}, $\se(A)=\{\qu\in\sigma_S(A)~~|~~R_\qu(A)\not\in\FF(\vr)\}.$ Therefore,
$\displaystyle\sigma_S(A)=\se(A)\cup\bigcup_{k\in\Z}\sigma_k^S(A),$ clearly with $\displaystyle\se(A)\cap\bigcup_{k\in\Z}\sigma_k^S(A)=\emptyset$.
\end{proof}
\section{conclusion}
We have studied the approximate S-point spectrum, Fredholm operators and essential S-spectrum of a bounded right linear operator on a right quaternionic Hilbert space $\vr$. The left multiplication defined on $\vr$ does not play a role in the approximate S-point spectrum and  in the Fredholm theory. However, in getting a unital Calkin algebra it has appeared and hence played a role in the S-essential spectrum. Also an interesting characterization to the S-spectrum is given in terms of the Fredholm operators and its index.\\

In the complex case, the exact energy spectrum is obtainable and known only for a handful of Hamiltonians, and for the rest only approximate spectrum is computed numerically. In these approximate theories the Fredholm operators and the essential spectra are used frequently. Further, these theories are successfully used in transport operator, integral operator and differential operator problems \cite{Ben, Gar, Jeri, Jeri2, kub, con}. In the same manner, the theories developed in this note may be useful in obtaining the approximate S-point spectrum or essential S-spectrum for quaternionic Hamiltonians and their small norm perturbations.  For quaternionic Hamiltonians and potentials see \cite{Ad, Veto} and the references therein.
\section{Acknowledments}
K. Thirulogasanthar would like to thank the FRQNT, Fonds de la Recherche  Nature et  Technologies (Quebec, Canada) for partial financial support under the grant number 2017-CO-201915. Part of this work was done while he was visiting the University of Jaffna to which he expresses his thanks for the hospitality.


\begin{thebibliography}{XXXX}
\bibitem{Ad} Adler, S.L., {\em Quaternionic Quantum Mechanics and Quantum Fields}, Oxford University Press, New York, 1995.

\bibitem{AC} Alpay, D., Colombo, F., Kimsey, D.P., {\em The spectral theorem for quaternionic unbounded normal operators based on the $S$-spectrum}, J. Math. Phys. {\bf 57} (2016), 023503.

\bibitem{Ben} Benharrat, M., {\em Comparison between the different definitions of the essential spectrum and applications}, Ph.D. thesis (2013), University of Oran.

\bibitem{Fab} Colombo, F., Sabadini, I., \textit{On Some Properties of the Quaternionic Functional Calculus}, J. Geom. Anal., {\bf 19} (2009), 601-627.

\bibitem{Fab1} Colombo, F., Sabadini, I., \textit{On the  Formulations of the Quaternionic Functional Calculus}, J. Geom. Phys., {\bf 60} (2010), 1490-1508.

\bibitem{Fab2} Colombo, F., Gentili, G., Sabadini, I., Struppa, D.C., {\em Non commutative functional calculus: Bounded operators}, Complex Analysis and Operator Theory, {\bf 4} (2010), 821-843.


\bibitem{NFC} Colombo, F., Sabadini, I., Struppa, D.C., {\em Noncommutative Functional Calculus}, Birkh\"auser Basel, 2011.

\bibitem{con} Conway, J.B., {\em A course in functional analysis}, 2nd Ed., Springer-Verlag, New York (1990).

\bibitem{Veto} De Leo, S., Giardino, S.,{\em Dirac solutions for quaternionic potentials}, J. Math. Phys.{\bf 55} (2014),
022301.

\bibitem{Tok} Diagana, T., \textit{Almost Automorphic Type and Almost Periodic Type Functions in Abstract Spaces}, Springer International Publishing : Imprint: Springer, 2013.


\bibitem{Fa} Fashandi, M., {\em Compact operators on quaternionic Hilbert spaces}, Ser. Math. Inform. {\bf 28} (2013), 249-256.

\bibitem{Gar} Garding, L., {\em Essential spectrum of Schr\"odinger operators}, J. Funct. Anal. {\bf 52} (1983), 1-10.

\bibitem{ghimorper} Ghiloni, R., Moretti, W. and Perotti, A., {\em Continuous slice functional calculus in quaternionic Hilbert spaces\/,} Rev. Math. Phys. {\bf 25} (2013), 1350006.


\bibitem{Goh} Gohberg, I., Goldberg, S., Kaashoek, M.A., {\em Basic classes of linear operators}, Springer, Basel AG (2003).

\bibitem{Jeri} Jeribi, A., {\em Spectral theory and applications of linear operators and block operator matrices}, Springer International Publishing, Switzerland (2015).

\bibitem{Jeri2} Jeribi, A., Mnif, M., {\em Essential spectra and application to transport equations}, Acta Applicandae Mathematica, {\bf 89} (2005), 155-176.

\bibitem{Kre}  Kreyszig, E. O., \textit{Introductory functional analysis with applications}, John Wiley \& Sons. Inc, 1978.

\bibitem{kub} Kubrusly, C.S., {\em Spectral theory of operators on Hilbert spaces}, Springer, New York (2012).

\bibitem{BT} Muraleetharan, B., Thirulogasanthar, K., {\em Deficiency Indices of Some Classes of Unbounded $\quat$-Operators},  Complex Anal. Oper. Theory (2017), 1-29. https://doi.org/10.1007/s11785-017-0702-4.

\bibitem{Mu} Muraleetharan, B, Thirulogasanthar, K., {\em Coherent state quantization of quaternions}, J. Math. Phys., {\bf 56} (2015), 083510.

\bibitem{Mus}  Muscat, J., \textit{Functional Analysis: An Introduction to Metric Spaces, Hilbert Spaces, and Banach Algebras}, Springer, 2014.

\bibitem{Rud} Rudin, W., \textit{Functional Analysis}, International Series in Pure and Applied Mathematics, McGraw-Hill Inc., New York, 1991.

\bibitem{Vis} Viswanath, K., {\em Normal operators on quaternionic Hilbert spaces}, Trans. Amer. Math. Soc. {\bf 162} (1971), 337-350.


\bibitem{Wil} Willians, V., {\em Closed Fredholm and semi-Fredholm operators, essential spectra and perturbations}, J. Funct. Anal. {\bf 20} (1975), 1-25.

\end{thebibliography}
\end{document}